\documentclass[journal]{IEEEtran}

\usepackage{amsmath,amssymb}
\usepackage{graphicx}
\graphicspath{{figures/}}
\usepackage{booktabs}
\usepackage{enumitem}
\usepackage{pifont}
\usepackage{accents}
\usepackage{xcolor}

\newcommand{\symbox}[1]{\makebox[0.8em][c]{#1}}
\newcommand{\cmark}{\symbox{\ding{51}}}
\newcommand{\xmark}{\symbox{\ding{55}}}

\newcommand\norm[2][1]{\left\|#2\right\|_{#1}}
\newcommand\abs[1]{\left|#1\right|}

\newcommand\difffrac[3][1]{
	\ifnum #1=1
		\frac{\mathrm{d} #2}{\mathrm{d} #3}
	\else
	\frac{{\mathrm{d}}^{#1} #2}{\mathrm{d} #3^{#1}}
	\fi
}

\newcommand\R{{\mathbb{R}}}
\newcommand\N{{\mathbb{N}}}

\newcommand{\Setminus}[2]{{\left.#1\middle\backslash #2\right.}}

\newcommand\st{{\text{s.t.}}}

\renewcommand\vector[1]{\boldsymbol{#1}}

\newcommand\vb{{\vector{b}}}

\newcommand\vd{{\vector{d}}}

\newcommand\vf{{\vector{f}}}
\newcommand\vg{{\vector{g}}}
\newcommand\vh{{\vector{h}}}

\newcommand\vq{{\vector{q}}}
\newcommand\vr{{\vector{r}}}

\newcommand\vx{{\vector{x}}}
\newcommand\vy{{\vector{y}}}

\newcommand\vA{{\vector{A}}}

\newcommand\vtau{{\vector{\tau}}}

\newcommand\vzero{{\vector{0}}}

\newcommand\mC{{\mathbf{C}}}

\newcommand\mM{{\mathbf{M}}}
\newcommand\mQ{{\mathbf{Q}}}

\newcommand\bx{{\mathbf{x}}}
\newcommand\by{{\mathbf{y}}}
\newcommand\bu{{\mathbf{u}}}

\newcommand\f{{\text{f}}}

\DeclareMathOperator*{\argmax}{argmax}
\DeclareMathOperator*{\argmin}{argmin}
\DeclareMathOperator*{\minmax}{min/max}

\makeatletter
\renewcommand{\dddot}{%
    \accentset{{\cc@style.\mkern-1.7mu.\mkern-1.7mu.}}}
\makeatother
\DeclareRobustCommand{\underbar}[1]{\kern0.3\fontdimen2\font\underaccent{\bar}{\kern-0.3\fontdimen2\font #1}}

\newcommand\RomanNum[1]{\uppercase\expandafter{\romannumeral #1}}

\allowdisplaybreaks[4]
\makeatletter
\def\fps@figure{!t}
\def\fps@table{!t}
\makeatother

\usepackage[ruled]{algorithm2e}
\usepackage{algorithmic}
\SetAlCapHSkip{0pt}
\usepackage{etoolbox}
\makeatletter
\patchcmd{\@algocf@start}{\raggedright}{}{}{}
\makeatother

\usepackage[colorlinks=true,allcolors=blue]{hyperref}
\usepackage{fontawesome}
\usepackage{multirow}
\usepackage{subcaption}
\usepackage{tikz}
\usetikzlibrary{calc}

\usepackage{cases}
\usepackage{mathtools}

\usepackage{amsthm}
\theoremstyle{definition}
\newtheorem{lemma}{Lemma}
\newtheorem{theorem}{Theorem}
\newtheorem{proposition}{Proposition}

\theoremstyle{definition}
\newtheorem{definition}{Definition}
\newtheorem{assumption}{Assumption}
\newtheorem{problem}{Problem}
\newtheorem*{remark}{Remark}
\newtheorem{example}{Example}

\usepackage{makecell}

\let\oldremark\remark
\let\endoldremark\endremark
\renewenvironment{remark}
  {\oldremark}
  {\hfill$\Diamond$\endoldremark}

\begin{document}

\title{Reachability-Augmented Dual Dynamic Programming for Optimal Path Parameterization}
\author{
    Yunan~Wang,
    Jizhou~Yan,
    Chuxiong~Hu,
    and~Zeyang~Li
\thanks{Corresponding author: Chuxiong Hu (e-mail: cxhu@tsinghua.edu.cn).}
}

\maketitle

\begin{abstract}
    Optimal path parameterization (OPP) is a fundamental problem for planning trajectories along a prescribed geometric path under kinodynamic constraints and task-dependent objectives. Among existing OPP works, it remains challenging to simultaneously achieve feasibility guarantees, general-objective optimality certificates, and online-compatible computation in a unified framework, especially for third-order OPP (OPP3) with non-convex constraints. This paper proposes reachability-augmented dual dynamic programming (RDDP), a state-grid-free and objective-aware dynamic programming (DP) framework for OPP. The key idea is to replace the relatively complete recourse assumption used in classical dual DP (DDP) with OPP-specific backward reachable sets, and then generate both value-function cuts and trial trajectories only inside these reachable sets. For convex and non-convex OPP, we prove global optimality and Karush-Kuhn-Tucker convergence of RDDP under OPP-specific conditions, respectively. Efficient instantiations are developed for second-order OPP (OPP2) and OPP3. Experiments show that RDDP achieves objective values comparable to convex-optimization baselines with speedups of approximately 15.3 and 4.9 times over the fastest of 4 tested solvers for OPP2 and OPP3, respectively. RDDP also achieves faster convergence than grid-based DP. Compared with reachability-analysis methods, RDDP retains the reachability mechanism while replacing local maximum-control propagation with value-function-guided control selection, thereby enabling objectives beyond traversal time. In summary, RDDP addresses a key capability gap in OPP by unifying certifiable general-objective optimization, reachability-based feasibility preservation, and online-compatible low-dimensional DP computation in a single framework.
\end{abstract}


\IEEEpeerreviewmaketitle

\section{Introduction}\label{sec:Introduction}

\IEEEPARstart{O}{ptimal} path parameterization (OPP) is a fundamental problem in robotics and optimal control. It aims to find an optimal time parameterization of a given geometric path under kinodynamic constraints and task-dependent objectives. Since the early studies in the 1980s \cite{bobrow1985time,shin1985minimum,shin1986dynamic}, OPP has been widely applied to manipulators \cite{verscheure2009time,debrouwere2013time,wang2022third}, humanoid robots \cite{pham2018new,hauser2014fast}, and manufacturing systems \cite{wang2026online,xu2025real}.

When the objective is to minimize the traversal time, OPP reduces to the well-known time-optimal path parameterization (TOPP) \cite{pham2018new}. By Pontryagin's maximum principle (PMP) \cite{pontryagin1987mathematical}, the optimal control of TOPP is bang-singular-bang, which pushes the velocity, acceleration, torque, and jerk to their limits. Despite the high motion efficiency, such constraint-saturating trajectories might induce undesirable behaviors such as vibration and tracking error in some scenarios. By trading traversal time with task-dependent criteria such as energy consumption and smoothness \cite{verscheure2009time,kaserer2019nearly}, general objectives can smooth the motion profile while maintaining high motion efficiency. An example is provided in Section \ref{subsec:real_world_experiment}.

For online robotic applications, an ideal OPP solver should satisfy the following four requirements:
\begin{enumerate}
    \item \textbf{Feasibility}: The generated trajectory should satisfy state and mixed state-control constraints, which are essential for the safety and reliability of robotic systems.
    \item \textbf{General-objective optimality}: Appropriate optimality certificates should be available for general objectives.
    \item \textbf{Online compatibility}: The solver should be computationally efficient, support early termination once a feasible solution is available, and allow further optimization with additional computation.
    \item \textbf{High-order capability}: Higher-order systems should be supported, especially third-order OPP (OPP3) which involves non-convex jerk constraints.
\end{enumerate}
It remains challenging to simultaneously achieve all these properties within a unified framework. This key capability gap is particularly pronounced for OPP3 due to non-convexity.

\begin{table*}[!t]
    \centering
    \caption{Comparison of different classes of OPP solvers under typical convex OPP2 and non-convex OPP3 settings.}
    \label{tab:related_works}
    \begin{tabular}{llllll}
        \toprule
        Method & Order & Objective Scope & Optimality & Feasibility & Efficiency \\ \midrule
        NI \cite{pham2014general} & 2 & Time only & \cmark\,(Global) & \cmark & High \\
        CO \cite{verscheure2009time,hauser2014fast} & 2 & General & \cmark\,(Global) & \cmark & Low \\
        DP \cite{shin1986dynamic} & 2 & General & \cmark\,(Grid-optimal) & \cmark\,(With fine grids) & Medium \\
        RA \cite{pham2018new} & 2 & Time only & \cmark\,(Global, without zero-inertia points) & \cmark & High \\
        \textbf{RDDP (Ours)} & 2 & General & \cmark\,(Global) & \cmark & Medium--High \\ \midrule
        NI \cite{pham2017structure} & 3 & Time only & \xmark & \xmark & N/A \\
        CO \cite{debrouwere2013time,wang2026online} & 3 & General & \cmark\,(KKT convergence) & \cmark & Low \\
        DP \cite{shin1986dynamic} & 3 & General & \cmark\,(Grid-optimal) & \cmark\,(With fine grids) & Extremely Low \\
        Greedy DP \cite{kaserer2019nearly} & 3 & General & \xmark & \xmark & Medium \\
        RA \cite{dio2025time} & 3 & Time only & \cmark\,(Empirically near optimal) & \cmark & High \\
        \textbf{RDDP (Ours)} & 3 & General & \cmark\,(KKT convergence) & \cmark & Medium \\ \bottomrule
    \end{tabular}
\end{table*}

\subsection{Related Works}\label{subsec:related_works}

\subsubsection{OPP solvers}\label{sub2sec:review_opp_solvers}

Existing OPP solvers can be categorized into three classes: numerical integration (NI), convex optimization (CO), and dynamic programming (DP). We classify reachability analysis (RA) approaches as DP, as they involve the backward propagation of reachable sets in a DP manner.

As the earliest approaches to TOPP \cite{shin1985minimum}, NI methods exploit the bang-singular-bang structure of the optimal control induced by PMP and reduce the decision space to a binary choice between maximum- and minimum-control arcs. They are efficient for TOPP2 with time-optimality guaranteed by the consistent maximum property. Despite numerical issues induced by dynamic singularity \cite{shiller1992computation}, robust TOPP2-NI implementations have been developed \cite{pham2014general}. However, a complete TOPP3-NI method is still lacking \cite{pham2017structure}. \textit{Objective-agnostic} NI methods also do not naturally extend to general-objective OPP, as the bang-singular-bang structure no longer holds.

CO approaches provide the strongest generic optimality certificates among existing OPP solvers. The standard optimization formulations for second-order OPP (OPP2) and OPP3 were established in \cite{verscheure2009time} and \cite{debrouwere2013time}, respectively. These formulations are highly expressive \cite{hauser2014fast}. For convex OPP, CO methods can guarantee global optimality; for non-convex OPP, CO methods converge to Karush-Kuhn-Tucker (KKT) solutions using sequential convex programming (SCP). Feasibility-preserving windowing has also been developed in online applications with long paths \cite{wang2026online}.
However, solving the full path-discretized problem can be computationally expensive, as thousands of decision variables and constraints are usually involved in a single optimization problem.

The above observation motivates DP studies: they decompose the full large-scale optimization problem into a sequence of low-dimensional subproblems. Standard DP methods discretize the state space into grid points \cite{shin1986dynamic} and generate the optimal grid sequence for general objectives based on Bellman's principle of optimality (BPO) \cite{bellman1952theory}. Refining the state grid improves solution quality, but computational cost is significantly increased, and the optimality convergence is slow.
Furthermore, OPP3 exhibits the curse of dimensionality due to the two-dimensional state space. Greedy pruning can reduce this cost \cite{kaserer2019nearly}, which may discard states needed for a feasible solution.

In contrast to grid-based DP approaches, RA provides a state-grid-free feasibility mechanism: backward reachable sets are propagated directly in the continuous state space. For convex TOPP2, reachable sets can be exactly computed as intervals \cite{pham2018new,consolini2019optimal}. For non-convex TOPP3, backward reachable sets are typically approximated by conservative inner polytopes \cite{dio2025time}. Once backward reachable sets are computed, the forward propagation locally maximizes the admissible control. This greedy policy is efficient for TOPP, but fails to minimize a general objective due to the absence of a cost-to-go model.

Table \ref{tab:related_works} summarizes the above review and illustrates the capability gap addressed in this paper. The computational efficiency is compared qualitatively within each order.

\subsubsection{Dual Dynamic Programming (DDP)}\label{sub2sec:review_dpp}

DDP \cite{read1990dual,pereira1991multi} is a state-grid-free DP method for multistage convex optimization, which partially alleviates the curse of dimensionality. DDP approximates the value function from below using affine optimality cuts over the continuous state space. The optimality cuts are repeatedly refined around the trial trajectory generated by the current approximate value function.

From an optimization perspective, DDP can be interpreted as nested Benders decomposition \cite{bnnobrs1962partitioning,birge1985decomposition}, which approximates convex value functions by affine cuts. Instead of solving the large-scale problem as a whole, DDP decomposes it into a sequence of subproblems while preserving global optimality convergence \cite{fullner2025stochastic} under standard assumptions.

A central feasibility assumption of standard DDP is the relatively complete recourse \cite{rockafellar1976stochastic}: every feasible current state admits a feasible continuation in future stages. Once such an assumption fails, feasibility cuts are introduced to exclude states that cannot be extended feasibly to the terminal set \cite{guigues2014sddp}, where an outer approximation of the recourse-feasible region is iteratively refined. This mechanism is effective for general multistage optimization problems if the recourse-feasible region is not explicitly available.

To the best of the authors' knowledge, however, DDP has not been introduced to robotic OPP so far. Directly applying classical DDP to OPP poses structural challenges. The next subsection examines these OPP-specific challenges and outlines the key ideas for addressing them.

\subsection{Key Ideas and Challenges}\label{subsec:key_ideas}

The key idea of this paper is to combine two complementary components: the reachability mechanism of RA for feasibility preservation and the cost-to-go model of DDP for \textit{objective-aware} control selection. The forward pass minimizes the stage cost plus the approximate value over admissible controls, rather than locally maximizing the control as in RA methods. Whereas conventional grid-based DP relies on a finite state grid, DDP uses value-function cuts to construct a state-grid-free approximation over the continuous state space. However, adapting DDP to OPP requires addressing two challenges.

The first difficulty lies in feasibility. In OPP, the terminal set is typically a singleton, and reachable sets can become lower-dimensional near the terminal stage. Thus, relatively complete recourse usually fails: a locally feasible state may have no feasible continuation to the terminal set. A standard alternative is to introduce feasibility cuts and refine them iteratively, but this can be computationally expensive and may not yield feasible trajectories in early iterations, thus being undesirable for online applications. This paper instead explicitly uses backward reachable sets as the restricted domains for value-function approximation and forward trajectory generation.

Another difficulty arises from non-convexity: affine cuts need not be lower bounds of a non-convex value function. Inspired by SCP, this paper embeds DDP into an inexact difference-of-convex (DC) decomposition framework to handle non-convex OPP. The analysis accounts for two OPP-specific issues: conservative reachable sets restrict the convexified feasible domain, and finite inner iterations leave quasi-optimal solutions for subproblems.

\subsection{Contributions and Comparison with Existing Methods}\label{subsec:contributions}

The central contribution is a reachability-augmented DDP (RDDP) framework for general-objective OPP, the comparison of which with existing methods is summarized in Table \ref{tab:related_works}. The following contributions cover its framework, theoretical guarantees, and efficient instantiations for OPP2 and OPP3.
\begin{enumerate}
    \item \textbf{A novel OPP framework augmenting DDP with RA.} To the best of the authors' knowledge, this paper introduces a DDP formulation to robotic OPP for the first time. To avoid relying on recourse assumptions or feasibility cuts in DDP that are unsuitable for OPP, this paper develops an OPP-specific RDDP framework. The key difference from a direct application of DDP is that RDDP explicitly uses backward reachable sets as recourse-feasible domains within which value-function cuts and trial trajectories are generated. Compared with RA, RDDP preserves reachability but replaces local maximum-control propagation with value-function-guided control selection, thereby enabling general objectives beyond traversal time. In contrast to conventional grid-based DP, RDDP does not require the generated trajectory to pass through any state-grid points. Compared with CO, RDDP avoids solving the full large-scale optimization problem by decomposing it into a sequence of low-dimensional subproblems.
    \item \textbf{Feasibility and optimality guarantees.} Since RDDP replaces the standard recourse assumption of DDP with explicit reachability, this paper establishes feasibility and optimality guarantees for RDDP under OPP-specific conditions. For convex OPP (COPP), the real cost is proved to converge to the global optimal value. Benefiting from the introduced reachable sets, each forward iteration of RDDP generates a feasible trajectory with a computable DDP optimality gap. This property is particularly important for online OPP, as it enables practical early termination under limited computational budgets. For general non-convex OPP (GOPP), RDDP is embedded into an inexact DC decomposition framework which accommodates both conservative reachable sets and finite-accuracy subproblem solutions. This framework guarantees feasibility and non-increasing real costs; under local exactness and refined Slater condition, and other requirements, the solution sequence additionally exhibits KKT convergence. Therefore, RDDP achieves a comparable optimality level to CO methods, i.e., global optimality and KKT convergence in convex and non-convex OPP, respectively.
    \item \textbf{Efficient instantiations for OPP2 and OPP3.} This paper develops concrete RDDP algorithms for the widely studied OPP2 and OPP3. For OPP2, exact interval reachable sets and value-function cuts are computed through two- and one-dimensional convex optimization, respectively. For OPP3, this paper develops an improved projection-based method with more accurate residual-area estimation and a searching-based method using low-dimensional linear programs to compute conservative polytopic reachable sets. In numerical experiments, RDDP achieves objective values comparable to CO baselines. Even relative to the fastest CO baseline among 4 different optimization solvers in each problem class, RDDP achieves speedups of approximately 15.3 and 4.9 times for OPP2 and OPP3, respectively. In summary, the developed RDDP can serve as a state-grid-free and objective-aware solver for general-objective OPP2 and OPP3 with feasibility and optimality for online robotic applications.
\end{enumerate}

\section{Problem Formulation and Preliminaries}\label{sec:formulation_preliminaries}

Section \ref{subsec:formulation_opp} formulates OPP as an optimization problem. Then, Section \ref{subsec:problem_statement} defines COPP and GOPP. As theoretical preliminaries, Section \ref{subsec:preliminary_dp} introduces the value function and backward reachable set underlying RDDP.

\subsection{Formulation of OPP}\label{subsec:formulation_opp}

Consider an $n$-dimensional robotic system with configuration $\vq\in\R^n$. A regular geometric path $\vq=\vq(s)$ of $\mathcal{C}^m$ smoothness is given, where $s\in[0,s_\text{f}]$ is the path parameter and the terminal parameter $s_\text{f}$ is fixed. Note that $s$ is not necessarily the arc length. The OPP problem is to find a strictly increasing time parameterization $s=s(t)$ that minimizes a user-specified objective under state and control constraints, where $t\in[0,t_\text{f}]$ is the time and the terminal time $t_\text{f}$ is free.

Following the standard process in \cite{verscheure2009time,debrouwere2013time}, the above description in the $t$-domain can be transformed into the $s$-domain, which is more convenient for numerical solution. We then discretize the parameter domain $[0,s_\text{f}]$ into a grid $0=s_0<s_1<\cdots<s_N=s_\text{f}$. Denote the state by $\vx(s_k)=\vx_k\in\R^{m-1}$ and take the control on $(s_k,s_{k+1})$ to be $u(s)\equiv u_k\in\R$. Specifically, $(\vx,u)$ is the discretization of $\big(\frac12\dot{s}^2,\frac{\mathrm{d}}{\mathrm{d}s}\big(\frac12\dot{s}^2\big),\dots,\frac{\mathrm{d}^{m-1}}{\mathrm{d}s^{m-1}}\big(\frac12\dot{s}^2\big)\big)$. For notational simplicity, denote $\bx_j:=(\vx_k)_{k=j}^N$ and $\bu_j:=(u_k)_{k=j}^{N-1}$. The resulting discrete optimization problem in the $s$-domain is given by
\begin{subequations}\label{eq:op_discretize}
    \begin{align}
        \min_{\bx_0,\bu_0} \quad &J = \sum_{k=0}^{N-1} L_k(\vx_k, u_k) + \Phi(\vx_N) \label{eq:op_discretize_objective}\\
        \st \quad & \vx_{k+1} = \vA_k\vx_k + \vb_k u_k, \,\,k=0,1,\ldots,N-1,\label{eq:op_discretize_dynamics} \\
        &\vf_k(\vx_k) \leq \vzero,  \,\,k=0,1,\ldots,N-1, \label{eq:op_discretize_state_constraints}\\
        &\vh_k(\vx_k, u_k) \leq \vzero, \,\,k=0,1,\ldots,N-1,\label{eq:op_discretize_mix_constraints}\\
        &\vx_0\in\mathcal{X}_0, \,\,\vx_N\in\mathcal{X}_\text{f}.
    \end{align}
\end{subequations}
For the direct transcription with $u(s)\equiv u_k$ \cite[Equation (20)]{wang2025time}, the exact state transition \eqref{eq:op_discretize_dynamics} is obtained with
\begin{equation}\label{eq:Ab}
    \vA_k=\left[\begin{array}{cccc}
        1 & \Delta_k & \cdots & \frac{\Delta_k^{m-2}}{(m-2)!} \\
        0 & 1 & \cdots & \frac{\Delta_k^{m-3}}{(m-3)!} \\
        \vdots & \vdots & \ddots & \vdots \\
        0 & 0 & \cdots & 1
    \end{array}\right],\,\,\vb_k=\left[\begin{array}{c}
        \frac{\Delta_k^{m-1}}{(m-1)!} \\
        \frac{\Delta_k^{m-2}}{(m-2)!} \\
        \vdots \\
        \Delta_k
    \end{array}\right],
\end{equation}
where $\Delta_k= s_{k+1}-s_{k}>0$. The continuous profile $\vx(s)$ is uniquely reconstructed analytically from the discrete solution. For OPP2 \cite[Section V-A]{verscheure2009time} and OPP3 \cite[Section 3.2]{wang2026online}, the maps $t\mapsto s$ and $s\mapsto t$ also admit closed-form expressions.

In this paper, all theoretical guarantees apply to the discrete formulation \eqref{eq:op_discretize} under the stated assumptions. The applicability of formulation \eqref{eq:op_discretize} to general OPP problems and its relationship to the continuous-time OPP are discussed in Appendix \ref{app:applicability}. Exclusion of zero-inertia points \cite{pham2014general} is not required.

\begin{assumption}\label{assum:feasibility_smoothness}
    Functions $L_k$, $\Phi$, $\vf_k$, and $\vh_k$ are locally Lipschitz continuous in their arguments for each $k$. The sets $\mathcal{X}_0$, $\mathcal{X}_\text{f}$, and $\{(\vx_k,u_k)\mid \vf_k(\vx_k)\leq\vzero,\,\vh_k(\vx_k,u_k)\leq\vzero\}$ are compact for each $k$.
\end{assumption}

\begin{remark}
    Most OPP problems of interest satisfy Assumption \ref{assum:feasibility_smoothness}. For example, the state and control are usually bounded by kinodynamic limits. The objective and constraint functions are usually locally Lipschitz continuous.

    Two exceptions need to be clarified. Denote the derivative of a variable $\bullet$ with respect to $t$ as $\dot{\bullet}=\frac{\mathrm{d}\bullet}{\mathrm{d}t}$. First, for stationary boundary conditions in OPP3, i.e., $\dot{s}=0$ at $s=0$ and $s=s_\text{f}$, the control $u=\frac{\dddot{s}}{\dot{s}}$
    can be unbounded even if the jerk limit \eqref{eq:jerk_constraint} is bounded. In this case, the zero-order hold of $\frac{\dddot{s}}{\dot{s}}$
    in the boundary interval can be replaced by a zero-order hold of $\dddot{s}$ \cite[Section 3.2.2]{wang2026online}, where the dynamic matrix and vector in \eqref{eq:Ab} are modified accordingly. Second, some objective functions are ill-defined at the boundary of the state space. For example, for some $k$, $L_k(\vx_k, u_k)=\frac{1}{\sqrt{2x_{k,1}}}$ is infinite at $x_{k,1}=0$. Since the optimal value is finite, an optimal solution must satisfy $x_{k,1}^\star>0$. To avoid the singularity, we consider a lower level set of $L_k$, such as $\{x_{k,1}\geq \varepsilon_{\text{sin}}\}$, with $\varepsilon_{\text{sin}}\in(0,x_{k,1}^\star)$. The problem remains feasible, and the original optimal solution remains optimal. So the above two exceptions can be handled by minor modifications of formulation \eqref{eq:op_discretize}, and thus do not violate Assumption \ref{assum:feasibility_smoothness}.
\end{remark}

\subsection{Problem Statement}\label{subsec:problem_statement}

This paper considers two classes of OPP \eqref{eq:op_discretize}, i.e., convex and non-convex problems. Both of the following problems are considered under Assumption \ref{assum:feasibility_smoothness}. Some common OPP2 problems are typically convex, whereas some typical OPP3 problems are non-convex due to the non-convex jerk constraints \eqref{eq:jerk_constraint}.

\begin{problem}[Convex optimal path parameterization, COPP]\label{problem:convex}
    Functions $L_k$, $\Phi$, $\vf_k$, and $\vh_k$ are convex in their arguments for each fixed $k$. Sets $\mathcal{X}_0$ and $\mathcal{X}_\text{f}$ are convex as well. Problem \eqref{eq:op_discretize} is feasible. The goal is to find a globally optimal solution to COPP \eqref{eq:op_discretize}.
\end{problem}

\begin{definition}\label{def:DC_decomposable}
    \cite{debrouwere2013time} Consider a function $f:\Omega\subset\R^n\to\R$ where $\Omega$ is a non-empty convex set. If $f$ can be expressed as the difference of two convex functions (DC) which are $\mathcal{C}^1$-continuous, then $f$ is called DC-decomposable. In other words, there exist two $\mathcal{C}^1$-continuous convex functions $f_1,f_2:\Omega\to\R$ such that $f(\vx)=f_1(\vx)-f_2(\vx)$ for all $\vx\in\Omega$. Given $\vx^\star\in\Omega$, we define the convexified function of $f$ at $\vx^\star$ as follows:
    \begin{equation*}
        \overline{f}(\vx;\vx^\star)=f_1(\vx)-\left[f_2(\vx^\star)+\nabla f_2(\vx^\star)^\top(\vx-\vx^\star)\right].
    \end{equation*}

    Consider a vector-valued DC-decomposable function $\vf=(f_i)_{i=1}^m:\Omega\subset\R^n\to\R^m$, i.e., each component function $f_i$ is DC-decomposable. We define the convexified function of $\vf$ at $\vx^\star$ as $\overline{\vf}(\vx;\vx^\star)=(\overline{f}_i(\vx;\vx^\star))_{i=1}^m$.

    A set $\mathcal{X}=\{\vx\in\R^n\mid \vf(\vx)\leq\vzero,\,\vg(\vx)=\vzero\}$ is called DC-decomposable  if $\vf$ is DC-decomposable and $\vg$ is affine. We define the convexified set of $\mathcal{X}$ at $\vx^\star$ as $\overline{\mathcal{X}}(\vx^\star)=\{\vx\in\R^n\mid \overline{\vf}(\vx;\vx^\star)\leq\vzero,\,\vg(\vx)=\vzero\}$.
\end{definition}

\begin{problem}[General optimal path parameterization, GOPP]\label{problem:general}
    Functions $L_k$, $\Phi$, $\vf_k$, and $\vh_k$ are DC-decomposable. Sets $\mathcal{X}_0$ and $\mathcal{X}_\text{f}$ are also DC-decomposable. The goal is to find a KKT solution to GOPP \eqref{eq:op_discretize}.
\end{problem}

\begin{remark}
    The smoothness assumption in GOPP is stronger than that in COPP for the well-definedness of the KKT condition. In practice, such a smoothness requirement is inherently satisfied by the analytical expressions of the objective and constraint functions. As for some special objectives, such as the total variation of the torque \eqref{eq:objective_torque_variation}, $L_k$ might be non-smooth at some points. In this case, $L_k$ can be approximated by a smooth variant, such as a smoothed maximum operator \cite{nesterov2005smooth} or a piecewise polynomial function.
\end{remark}

Both COPP and GOPP have been well addressed in CO-based literature \cite{verscheure2009time,debrouwere2013time,wang2026online} by leveraging off-the-shelf general-purpose solvers. As analyzed in Section \ref{sec:Introduction}, this paper focuses on developing an efficient objective-aware DP framework for solving OPP, where the large-scale optimization problem \eqref{eq:op_discretize} should be decomposed into a sequence of low-dimensional local subproblems.

\subsection{Reachable Set and Value Function}\label{subsec:preliminary_dp}

RDDP is built on two structural components of DP \cite{bellman1952theory}, i.e., the backward reachable set and the value function. The former provides feasibility preservation, while the latter enables general-objective optimization.

\subsubsection{Backward Reachable Set} To guarantee feasibility, the backward reachable set \cite{bokanowski2010reachability} is introduced as follows.

\begin{definition}\label{def:reachable_set}
    The backward reachable set $\mathcal{B}_k$ at $s_k$ is the set of states $\vx_k$ from which there exists a feasible trajectory that can reach the terminal set $\mathcal{X}_\text{f}$, i.e., $\forall k=0,1,\dots,N$,
    \begin{align}
        \mathcal{B}_k=\big\{&\vx_k\in\R^{m-1}\mid\exists \vx_{k+1}, \ldots, \vx_N, u_k, \ldots, u_{N-1}, \notag\\
        &\text{ s.t. } \forall k\leq j<N,\,\vx_{j+1} = \vA_j\vx_j+\vb_j u_j, \notag\\
        &\qquad\vf_j(\vx_j) \leq \vzero,\,\vh_j(\vx_j, u_j) \leq \vzero,\,\vx_N \in \mathcal{X}_\text{f} \big\}.\label{eq:backward_reachable_definition}
    \end{align}
    By convention, the terminal reachable set is $\mathcal{B}_N=\mathcal{X}_\text{f}$.
\end{definition}

\begin{figure}[!]
    \centering
    \input{figures/dp.tikz}
    \caption{Structural components of RDDP. (a) The backward reachable set $\mathcal{B}_k$ at $s_k$ is propagated backward from the terminal set $\mathcal{X}_\text{f}$ at $s_N$. The forward propagation based on $\mathcal{B}_k$ provides a sufficient and necessary condition for feasibility. (b) The value function $V_k^\star$ is propagated backward from the terminal cost $\Phi$ at $s_N$. The forward propagation based on $V_k^\star$ provides a sufficient and necessary condition for optimality.}
    \label{fig:dp}
\end{figure}

According to Definition \ref{def:reachable_set}, for $k<N$, the backward reachable set $\mathcal{B}_k$ satisfies the following backward recursion:
\begin{align}
    \mathcal{B}_k=\big\{&\vx_k\in\R^{m-1}\mid\exists u_k,\text{ s.t. } \vh_k(\vx_k, u_k) \leq \vzero, \notag\\
    &\vf_k(\vx_k)\leq\vzero,\,\vx_{k+1}=\vA_k\vx_k+\vb_k u_k\in\mathcal{B}_{k+1}\big\}.\label{eq:recursion_backward}
\end{align}
For each state $\vx_k \in \mathcal{B}_k$, the admissible control set is
\begin{align}
    \mathcal{C}_k(\vx_k) =\big\{&u_k \in \R \mid \vh_k(\vx_k, u_k) \leq \vzero, \notag\\
    &\vx_{k+1}=\vA_k\vx_k+\vb_k u_k\in\mathcal{B}_{k+1}\big\}.\label{eq:admissable_control_set}
\end{align}
Then, \eqref{eq:recursion_backward} means that $\vx_k\in\mathcal{B}_k$ if and only if $\vx_k$ satisfies the state constraints \eqref{eq:op_discretize_state_constraints} and $\mathcal{C}_k(\vx_k)\neq\varnothing$.

As illustrated in Fig. \ref{fig:dp}(a), \eqref{eq:recursion_backward} computes each $\mathcal{B}_k$ by backward propagation (BP) from $\mathcal{B}_{k+1}$ and the local constraints on $(\vx_k,u_k)$ at $s_k$. Once all $\mathcal{B}_k$ are computed, feasible trajectories can be generated by forward propagation (FP). Starting from an initial state $\vx_0\in\mathcal{B}_0\cap\mathcal{X}_0$, a control $u_k\in\mathcal{C}_k(\vx_k)\neq\varnothing$ is chosen to get $\vx_{k+1}=\vA_k\vx_k+\vb_k u_k\in\mathcal{B}_{k+1}$ for each $k$. In this way, $\forall k$, the state $\vx_k$ is guaranteed to be in $\mathcal{B}_k$.

\begin{remark}
    For general optimal control problems, $\mathcal{B}_k$ has no analytical expression and is computationally expensive \cite{bokanowski2010reachability}. The specific structure of second- (COPP2) and third-order COPP (COPP3) permits its efficient evaluation and representation \cite{pham2018new,dio2025time}. More details are given in Section \ref{sec:ardp_opp2_opp3}.
\end{remark}

\subsubsection{Value Function}

The value function \cite{bellman1962applied} $V_k^\star(\vx_k)$ is introduced to characterize optimality as follows.

\begin{definition}
    The value function $V_k^\star(\vx_k)$ is defined as the optimal cost-to-go from state $\vx_k$ at $s_k$, i.e., $V_k^\star:\mathcal{B}_k \to \R$,
    \begin{subequations}\label{eq:def_value_function}
        \begin{align}
            V_k^\star(\vx_k) = \min_{\bx_{k+1},\bu_k}&\sum_{j=k}^{N-1} L_j(\vx_j, u_j) + \Phi(\vx_N) \label{eq:def_value_function_objective}\\
            \st\,\,&\vx_{j+1} = \vA_j\vx_j + \vb_j u_j,\,\,j\geq k,\\
            &\vf_j(\vx_j) \leq \vzero,\,\,j> k,\\
            &\vh_j(\vx_j, u_j) \leq \vzero,\,\,j\geq k,\\
            &\vx_N \in \mathcal{X}_\text{f}.
        \end{align}
    \end{subequations}
    By convention, the terminal value function $V_N^{\star}(\vx_N)$ is defined as $V_N^{\star}(\vx_N)=\Phi(\vx_N)$ for $\vx_N\in\mathcal{X}_\text{f}=\mathcal{B}_N$.
\end{definition}

Bellman's optimality equation implies that for $k<N$,
\begin{equation}\label{eq:bellman_value_function}
    V_k^\star(\vx_k) = \min_{u_k\in\mathcal{C}_k(\vx_k)} L_k(\vx_k, u_k) + V_{k+1}^\star(\vA_k\vx_k+\vb_k u_k).
\end{equation}

As shown in Fig. \ref{fig:dp}(b), the value function $V_k^\star$ can also be computed via BP recursively from $s_N$ to $s_0$ by \eqref{eq:bellman_value_function}. Once all $V_k^\star$ at each state $\vx_k$
are computed, one can generate a globally optimal trajectory by FP from $s_0$ to $s_N$. Specifically, an initial state is chosen by $\vx_0^\star\in \argmin_{\vx_0\in\mathcal{B}_0\cap \mathcal{X}_0}V_{0}^\star(\vx_0)$. For each $k$, an optimal control is selected by $u_k^\star \in \argmin_{u_k\in\mathcal{C}_k(\vx_k^\star)}\big[ L_k(\vx_k^\star, u_k) + V_{k+1}^\star(\vA_k\vx_k^\star+\vb_k u_k)\big]$ to get $\vx_{k+1}^\star=\vA_k\vx_k^\star+\vb_k u_k^\star$ until $s_N$. The global optimality is guaranteed by the BPO \cite{bellman1962applied}.

\section{RDDP for COPP with Global Optimality}\label{sec:RDDP_COPP}

To address the intractability of ideal DP in Section \ref{subsec:preliminary_dp}, this section proposes the RDDP framework to solve COPP, i.e., the convex version of OPP defined in Problem \ref{problem:convex}. In contrast to the classical DDP where the relatively complete recourse assumption or feasibility cuts are used to handle the feasibility, RDDP explicitly uses the backward reachable set to guarantee the feasibility of forward trajectories, as presented in Section \ref{subsec:rddp_process}. Under such OPP-specific conditions, RDDP is proved to converge to the global optimum of COPP in Section \ref{subsec:optimality_adrp_copp}. All proofs are provided in Appendix \ref{app:proofs}.

\subsection{RDDP Process}\label{subsec:rddp_process}

We first briefly introduce the RDDP process. Before the first loop, RDDP computes the backward reachable sets, which identify the states from which the terminal set can still be reached. These sets remain fixed across all loops of RDDP.

As illustrated in Fig. \ref{fig:ardp_one_loop}(a), each RDDP loop consists of a backward cut-generation pass and a forward trial pass. For the $p$-th loop, the backward pass works from the terminal stage toward the start. At each sampled reachable state $\vx_k^i\in\mathcal{B}_k$, it constructs an affine lower cut of the value function. All sampled states $\vx_k^i$ form the cut-generation set $\mathcal{S}_k^{p}$, and all lower cuts collectively form the piecewise linear approximate value function $\mathcal{V}_k^p$. In the forward pass, it starts from the initial set and selects the optimal control $u_k^{p\star}$ for the approximate value function $\mathcal{V}_k^p$, while keeping each successor $\vx_{k+1}$ in the next reachable set $\mathcal{B}_{k+1}$. The resulting trial trajectory $\bx_0^{p\star}=(\vx_k^{p\star})_{k=0}^N$, $\bu_0^{p\star}=(u_k^{p\star})_{k=0}^{N-1}$ is feasible, and its states $\vx_k^{p\star}$ are added to the next cut-generation set $\mathcal{S}_k^{p+1}$. These added samples refine the value approximation, and they do not restrict the trial trajectory to a state grid.

The following properties support this construction.
\begin{lemma}\label{lemma:convexity_reachset_valfunc}
    In Problem \ref{problem:convex} under Assumption \ref{assum:feasibility_smoothness}, every backward reachable set $\mathcal{B}_k$ is convex and compact, whereas the value function $V_k^\star$ is convex in $\mathcal{B}_k$. The admissible control set $\mathcal{C}_k(\vx_k)\subset\R$ is convex and compact for each $\vx_k\in\mathcal{B}_k$; hence, $\mathcal{C}_k(\vx_k)$ is either a closed interval or a singleton. Specifically, if \eqref{eq:op_discretize_state_constraints}--\eqref{eq:op_discretize_mix_constraints} are linear inequalities and $\mathcal{X}_\text{f}$ is a polytope, then $\mathcal{B}_k$ is a polytope.
\end{lemma}

\begin{figure}[!t]
    \centering
    \input{figures/ardp_one_loop.tikz}
    \caption{RDDP process for COPP. (a) The dependence relationship in RDDP. (b) The relation between the real value function $V_k^\star$, the approximate value function $\mathcal{V}_k^p$, and the one-step approximate value function $\widehat{\mathcal{V}}_k^p$. The DDP optimality gap $\varepsilon_\text{gap}^k$ is defined in \eqref{eq:dual_gap} and \eqref{eq:dual_gap_k}.}
    \label{fig:ardp_one_loop}
\end{figure}

The details of the backward and forward passes are presented as follows.

\subsubsection{Backward Cut-Generation Pass}\label{sub2sec:backward_cut_pass}

At the beginning of the $p$-th loop, assume that a cut-generation set $\mathcal{S}_k^p=\{\vx_k^{i}\}_{i=1}^{I_k^p}\subset\mathcal{B}_k$ is available for each $k$, where $I_k^p\in\N^*$ is the sample number. Specifically, for $p=0$, these points can be sampled by a coarse grid or randomly. For $p>0$, the update of the cut-generation set will be analyzed in Section \ref{sub2sec:refinement}.

At the $p$-th loop, the backward pass aims to construct a tractable cut-based approximation $\mathcal{V}_k^p$ of the real value function $V_k^\star$ at $s_k$. Let the terminal value be
\begin{equation}\label{eq:approximation_terminal_value}
    \mathcal{V}_N^p=\Phi=V_N^\star.
\end{equation}
We aim to compute the one-step approximate value function $\widehat{\mathcal{V}}_k^p(\vx_k):\mathcal{B}_k\to\R$, where
\begin{equation}\label{eq:approximate_value_function_backward}
    \widehat{\mathcal{V}}_k^p(\vx_k) = \min_{u_k\in\mathcal{C}_k(\vx_k)} L_k(\vx_k, u_k) + \mathcal{V}_{k+1}^p(\vA_k\vx_k+\vb_k u_k).
\end{equation}
However, $\widehat{\mathcal{V}}_k^p$ remains intractable since there are uncountably many state points $\vx_k$ in $\mathcal{B}_k$. Section \ref{subsec:optimality_adrp_copp} will prove the convexity of $\widehat{\mathcal{V}}_k^p$. A lower approximation $\mathcal{V}_k^p$ of $\widehat{\mathcal{V}}_k^p$ is defined by optimality cuts, i.e., supporting hyperplanes of $\widehat{\mathcal{V}}_k^p$, at the cut-generation set $\mathcal{S}_k^p$. Specifically, $\forall \vx_k\in\mathcal{B}_k$, let
\begin{equation}\label{eq:support_plane}
    \mathcal{V}_k^p(\vx_k) = \max_{i=1,\ldots,I_k^p} \left\{\widehat{\mathcal{V}}_k^p(\vx_k^i) + \vg_k^{i\top}(\vx_k-\vx_k^i)\right\},
\end{equation}
where the subgradient $\vg_k^i\in\partial\widehat{\mathcal{V}}_k^p(\vx_k^i)$ is chosen such that $\norm[2]{\vg_k^i}\leq C_k^V$, with $C_k^V>0$ independent of $p$ and the cut-generation sets. The existence of such constants and subgradient selections will be proved in Proposition \ref{prop:Lipschitz}. Throughout this paper, subgradients of functions on $\mathcal{B}_k$ are understood as convex subgradients of their extended-real extensions that equal $\infty$ outside $\mathcal{B}_k$. Note that $\mathcal{C}_k(\vx_k^i)$ is an interval or singleton for each sample $\vx_k^i$. Since $\mathcal{V}_{k+1}^p$ is convex, $\widehat{\mathcal{V}}_k^p(\vx_k^i)$ can be efficiently computed through one-dimensional convex optimization \eqref{eq:approximate_value_function_backward} for each sample $\vx_k^i$.

As summarized in Fig. \ref{fig:ardp_one_loop}(b) and proved in Section \ref{subsec:optimality_adrp_copp}, it holds pointwise in $\mathcal{B}_k$ that $\mathcal{V}_k^p\leq \widehat{\mathcal{V}}_k^p\leq V_k^\star$. Therefore, the tractable cut-approximation $\mathcal{V}_k^p$ also serves as a lower bound of the real value function $V_k^\star$.

\subsubsection{Forward Trial Pass}\label{sub2sec:refinement}

In the forward trial pass for the $p$-th loop, the trial trajectory $\bx_0^{p\star}$ is generated under the guidance of the approximate value function $\mathcal{V}_k^p$, as shown in Fig. \ref{fig:ardp_one_loop}(a). Assume that $\mathcal{V}_k^p$ has been constructed in the backward pass. The trial trajectory is first generated as follows:
\begin{subequations}\label{eq:forward_refine}
    \begin{align}
        &\vx_0^{p\star}\in \argmin_{\vx\in\mathcal{B}_0\cap \mathcal{X}_0} \mathcal{V}_0^p(\vx),\label{eq:refine_x0}\\
        &u_k^{p\star} \in \argmin_{u\in\mathcal{C}_k(\vx_k^{p\star})} L_k(\vx_k^{p\star}, u) + \mathcal{V}_{k+1}^p(\vA_k\vx_k^{p\star}+\vb_k u),\label{eq:refine_uk}\\
        &\vx_{k+1}^{p\star} = \vA_k\vx_k^{p\star}+\vb_k u_k^{p\star}.
    \end{align}
\end{subequations}
\eqref{eq:approximate_value_function_backward} and \eqref{eq:forward_refine} imply that for each $k$,
\begin{equation}\label{eq:widehat_V_L_V_refine}
    \widehat{\mathcal{V}}_k^p(\vx_k^{p\star}) = L_k(\vx_k^{p\star}, u_k^{p\star}) + \mathcal{V}_{k+1}^p(\vx_{k+1}^{p\star}).
\end{equation}
Then, the trial trajectory $\bx_0^{p\star}$ is added to the cut-generation set for the next loop. We require
\begin{equation}\label{eq:ardp_more_sample}
    \mathcal{S}_k^{p+1}\supset\{\vx_k^{j\star}\}_{j=0}^p.
\end{equation}
Note that one can also add more cut-generation points around $\vx_k^{p\star}$ to further refine the approximation. Based on the updated $\mathcal{S}_k^{p+1}$, the backward cut-generation pass is performed to get $\mathcal{V}_k^{p+1}$, and then the forward trial pass is performed again to get $\bx_0^{(p+1)\star}$. The above process is repeated until the trial trajectory converges or the maximum number of loops is reached.

\begin{remark}
    For each $k$, $u_k^{p\star}$ and $\vx_{k+1}^{p\star}$ are obtained by solving the one-dimensional convex optimization problem \eqref{eq:refine_uk}. According to \eqref{eq:refine_x0}, a convex optimization problem of dimension no greater than $m-1$ is solved to get $\vx_0^{p\star}$. In practice, the initial state $\vx_0$ is usually specified, i.e., $\mathcal{X}_0=\{\vx_0\}$ is fixed. Due to the feasibility of COPP \eqref{eq:op_discretize} in Problem \ref{problem:convex}, $\vx_0\in\mathcal{B}_0$ holds. In this case, the initial state is directly set as $\vx_0^{p\star}=\vx_0$ without solving any optimization problem.
\end{remark}

\begin{algorithm}[!t]
    \caption{RDDP for COPP (optimality convergence).}
    \label{alg:ardp_copp}
    \begin{algorithmic}[1]
        \REQUIRE COPP \eqref{eq:op_discretize}; a user-defined sampling scheme of cut-generation sets; a convergence criterion.
        \ENSURE A feasible trajectory $(\bx_0^\star,\bu_0^\star)$ and its certified DDP optimality gap $\varepsilon_{\text{gap}}$.
        \STATE Initialize $p\leftarrow 0$.
        \STATE Initialize the terminal backward reachable set $\mathcal{B}_N=\mathcal{X}_\text{f}$.
        \REPEAT
            \STATE Set $\mathcal{V}_N^p\gets\Phi$.
            \FOR{$k\leftarrow N-1, N-2, \ldots, 0$}
                \IF{$p=0$}
                    \STATE Compute the backward reachable set $\mathcal{B}_k$ based on $\mathcal{B}_{k+1}$ via BP \eqref{eq:recursion_backward}.
                    \STATE Sample $\mathcal{S}_k^0=\{\vx_k^i\}_{i=1}^{I_k^p}\subset\mathcal{B}_k$ with $I_k^p\geq1$.
                \ELSE
                    \STATE Update $\mathcal{S}_k^p\supset\{\vx_k^{j\star}\}_{j=0}^{p-1}$.
                \ENDIF
                \STATE Compute the one-step approximate value $\widehat{\mathcal{V}}_k^p(\vx_k^i)$ and a subgradient $\vg_k^i$ at each sample $\vx_k^i\in\mathcal{S}_k^p$ based on $\mathcal{V}_{k+1}^p$ via BP \eqref{eq:approximate_value_function_backward}.
                \STATE Approximate the value function based on optimality cuts at $\mathcal{S}_k^p$ and get $\mathcal{V}_k^p:\mathcal{B}_k\to\R$ via \eqref{eq:support_plane}.
            \ENDFOR
            \STATE Find $\vx_0^{p\star}\in\mathcal{B}_0\cap \mathcal{X}_0$ to minimize $\mathcal{V}_0^p$ via \eqref{eq:refine_x0}.
            \FOR{$k\leftarrow 0, 1, \ldots, N-1$}
                \STATE Find the trial control $u_k^{p\star}$ to minimize $L_k(\vx_k^{p\star}, u) + \mathcal{V}_{k+1}^p(\vA_k\vx_k^{p\star}+\vb_k u)$ via \eqref{eq:refine_uk}.
                \STATE Get $\vx_{k+1}^{p\star} = \vA_k\vx_k^{p\star}+\vb_k u_k^{p\star}$.
            \ENDFOR
            \STATE Compute the real cost \eqref{eq:op_discretize_objective} of $\{\vx_k^{p\star}\}_{k=0}^{N}$ and its optimality gap, update the historical best trial trajectory.
            \STATE Let $p\gets p+1$.
        \UNTIL{the convergence criterion is satisfied.}
        \STATE Let $\{\vx_k^{\star}\}_{k=0}^{N}$, $\{u_k^{\star}\}_{k=0}^{N-1}$ be the best trial trajectory with the minimum real cost, and return its optimality gap.
    \end{algorithmic}
\end{algorithm}

\subsubsection{Algorithm and Complexity}
The proposed RDDP is summarized in Algorithm \ref{alg:ardp_copp}, where significant algorithmic flexibility is offered to allow users' tailored configurations. First, the user can compute the bidirectional reachable set and sample state points within it, which can improve the sample efficiency and achieve higher approximation accuracy. This strategy will be applied in COPP2 in Section \ref{subsec:ardp_copp2}. Second, the user can choose different sampling schemes for each loop. The only requirement is that all trial trajectories of previous loops, i.e., $\{\vx_k^{j\star}\}_{j=0}^{p-1}$, should be included in the cut-generation set $\mathcal{S}_k^p$ for each $p,k$. Third, the approximate value $\mathcal{V}_k^p(\vx_k^i)$ and its selected subgradient might keep unchanged across different loops, which can be reused to save computational costs. Finally, we recommend defining the convergence criterion based on the DDP optimality gap, which will be introduced in Theorem \ref{thm:near_optimality_gap}.

We analyze the complexity of RDDP iterations for COPP, excluding that of reachable sets, which are discussed in Section \ref{sec:ardp_opp2_opp3}. Let $\mathcal{O}(T_\text{1D})$ uniformly bound the cost of solving a one-dimensional convex subproblem in \eqref{eq:approximate_value_function_backward} or \eqref{eq:refine_uk}, which may depend on the stage constraints and lower cuts. For a prescribed initial state, the backward and forward passes at stage $k$ of RDDP loop $p$ require at most $I_k^p$ and 1 such solve, respectively. Therefore, the time complexity is $T_\text{COPP}=\mathcal{O}\left(T_\text{1D}\sum_{p=0}^{N_\text{inner}-1}\sum_{k=0}^{N-1}(I_k^p+1)\right)=\mathcal{O}\left(NN_\text{inner}N_\text{cut}T_\text{1D}\right)$. Here, $N_\text{inner}$ is the number of RDDP loops and $N_\text{cut}:=\max_{k,p}I_k^p$. Similarly, the space complexity is $\mathcal{O}\left(mNN_\text{cut}\right)$.

\subsection{Global Optimality Guarantee}\label{subsec:optimality_adrp_copp}

This subsection proves the convergence and optimality of RDDP for COPP under the OPP-specific condition. Before that, we first investigate the Lipschitz continuity of $\widehat{\mathcal{V}}_k^p$ in \eqref{eq:approximate_value_function_backward} and the existence of subgradients bounded by the Lipschitz constant for well-definedness. Different from the standard DDP under the relatively complete recourse assumption, this paper only requires the following regularity condition on the admissible control set.

\begin{assumption}\label{assum:Lipschitz}
    For each $k$, the admissible control set $\mathcal{C}_k(\vx_k)$ can be represented as
    \begin{equation*}
    \mathcal{C}_k(\vx_k) = \{u\in\R \mid \underbar{u}_k(\vx_k)\leq u\leq \bar{u}_k(\vx_k)\},
    \end{equation*}
    where $\underbar{u}_k$, $-\bar{u}_k$ are convex. Assume that $\underbar{u}_k$, $\bar{u}_k$ are locally Lipschitz continuous in $\mathcal{B}_k$, including its boundary points.
\end{assumption}

\begin{remark}
    Assumption \ref{assum:Lipschitz} is typically satisfied in OPP. In the widely studied OPP2 and OPP3, the kinodynamic constraints and their convexified variants are often affine in state and control \cite[Appendix A.2]{wang2026online}. In this case, $\underbar{u}_k$ and $\bar{u}_k$ are piecewise linear and thus locally Lipschitz continuous. Only in degenerate cases can the local Lipschitz continuity fail in the boundary of $\mathcal{B}_k$. For example, if $u_k\leq\bar{u}_k(\vx_k)=\sqrt{1-x_{k,1}^2}$, then $\bar{u}_k$ is not locally Lipschitz continuous at boundary points $x_{k,1}=1$. For such an uncommon case, one can consider the subproblem with a slightly contracted feasible domain to exclude the singularity. Therefore, Assumption \ref{assum:Lipschitz} is not restrictive for the OPP instances considered in this paper.
\end{remark}

\begin{proposition}\label{prop:Lipschitz}
    In Problem \ref{problem:convex} under Assumptions \ref{assum:feasibility_smoothness}--\ref{assum:Lipschitz}, there exist constants $\{C_k^V\}_{k=0}^N\subset(0,\infty)$, independent of $p$ and the cut-generation sets, such that for every $p\geq0$ and any nonempty finite cut-generation sets $\mathcal{S}_k^p\subset\mathcal{B}_k$, the backward recursion \eqref{eq:approximation_terminal_value}--\eqref{eq:support_plane} can be carried out with subgradients $\vg_k^i\in\partial\widehat{\mathcal{V}}_k^p(\vx_k^i)$ satisfying $\norm[2]{\vg_k^i}\leq C_k^V$ at every sample. For any such choices, $\widehat{\mathcal{V}}_k^p$, $\mathcal{V}_k^p$, and $V_k^\star$ are convex and $C_k^V$-Lipschitz continuous on $\mathcal{B}_k$.
\end{proposition}

Proposition \ref{prop:Lipschitz} establishes well-definedness for RDDP. In the following analysis, Proposition \ref{prop:bound_of_value_function} and Theorem \ref{thm:near_optimality_gap} recover a DDP-like upper-lower-bound convergence structure under OPP-specific feasibility and regularity conditions. In contrast to the relatively complete recourse assumption in classical DDP, RDDP explicitly uses the backward reachable sets $\mathcal{B}_k$ as the domain of the value functions and enforces recourse feasibility through the constraint $\vx_{k+1}\in\mathcal{B}_{k+1}$.

\begin{proposition}\label{prop:bound_of_value_function}
    Adopt the setting of Proposition \ref{prop:Lipschitz}. In the backward cut-generation pass of RDDP, for each $k<N$ and $p\geq0$, any feasible solution $(\bx_k,\bu_k)$ of COPP \eqref{eq:op_discretize} satisfies
    \begin{equation}\label{eq:bound_of_value_function}
        \mathcal{V}_k^p(\vx_k) \leq\widehat{\mathcal{V}}_k^p(\vx_k) \leq V_k^\star(\vx_k) \leq J_k(\bx_k,\bu_k),
    \end{equation}
    where $J_k(\bx_k,\bu_k)$ is the real cost in \eqref{eq:op_discretize_objective}, i.e.,
    \begin{equation*}
        J_k(\bx_k,\bu_k) = \sum_{j=k}^{N-1} L_j(\vx_j, u_j) + \Phi(\vx_N).
    \end{equation*}
\end{proposition}

In \eqref{eq:bound_of_value_function}, the piecewise linear approximation $\mathcal{V}_k^p$ and the real cost $J_k$ are both computationally tractable, whereas the real value function $V_k^\star$ and the one-step approximation $\widehat{\mathcal{V}}_k^p$ are intractable. As illustrated in Fig. \ref{fig:ardp_one_loop}(b), the key theoretical mechanism is a tractable upper-lower bound certificate: the real cost $J_k$ provides an upper bound on the optimal value, while the cut-based approximation $\mathcal{V}_k^p$ provides a lower bound.

\begin{theorem}\label{thm:near_optimality_gap}
    Adopt the notation and setting of Proposition \ref{prop:Lipschitz}. For the $p$-th loop of RDDP, consider the feasible trial trajectory $\bx_0^{p\star}=(\vx_k^{p\star})_{k=0}^N$, $\bu_0^{p\star}=(u_k^{p\star})_{k=0}^{N-1}$ generated in the forward pass. Denote the globally optimal value of COPP \eqref{eq:op_discretize} by $J^\star$. Then, the optimality gap is bounded by
    \begin{equation*}
        0\leq J_0(\bx_0^{p\star},\bu_0^{p\star}) - J^\star \leq \varepsilon_\text{gap}(\bx_0^{p\star},\bu_0^{p\star}),
    \end{equation*}
    where the \textit{DDP optimality gap} $\varepsilon_\text{gap}$ is defined as
    \begin{equation}\label{eq:dual_gap}
        \varepsilon_\text{gap}(\bx_0^{p\star},\bu_0^{p\star})=J_0(\bx_0^{p\star},\bu_0^{p\star}) - \mathcal{V}_0^p(\vx_0^{p\star}).
    \end{equation}
\end{theorem}

Theorem \ref{thm:near_optimality_gap} provides a stopping criterion for Algorithm \ref{alg:ardp_copp}, i.e., the DDP optimality gap $\varepsilon_\text{gap}$ is smaller than a user-defined threshold. Finally, the global optimality convergence of RDDP is presented as follows.

\begin{theorem}\label{thm:adrp_optimal}
    Adopt the notation and setting of Theorem \ref{thm:near_optimality_gap}. As $p\to\infty$, we have
    \begin{equation}\label{eq:convergence_condition}
        \varepsilon_\text{gap}(\bx_0^{p\star},\bu_0^{p\star})\to0,
        \, J_0(\bx_0^{p\star},\bu_0^{p\star})\to J^\star,
    \end{equation}
    where $J^\star$ is the global optimal value. Moreover, for any optimal solution $(\bx_0^{\star},\bu_0^{\star})$, as $p\to\infty$, $\varepsilon_\text{gap}(\bx_0^{\star},\bu_0^{\star})\to0$.
\end{theorem}

Based on Theorem \ref{thm:adrp_optimal}, the proposed RDDP can be regarded as a principled alternative to conventional CO methods for COPP. Note that RDDP attains the same global optimality while replacing the large-scale optimization problem \eqref{eq:op_discretize} by a sequence of low-dimensional ones, thereby enabling substantially more efficient computation, as verified in Section \ref{subsec:numerical_experiments}.

\begin{remark}
    Proposition \ref{prop:bound_of_value_function} and Theorem \ref{thm:adrp_optimal} provide a key insight which enables adaptive refinement in RDDP. Specifically, we define a tractable one-step DDP optimality gap as follows:
    \begin{equation}\label{eq:dual_gap_k}
        \varepsilon_\text{gap}^k(\vx_k^{p\star}, u_k^{p\star}) = L_k(\vx_k^{p\star}, u_k^{p\star}) + \mathcal{V}_{k+1}^p(\vx_{k+1}^{p\star}) - \mathcal{V}_k^p(\vx_k^{p\star}).
    \end{equation}
    If $\varepsilon_\text{gap}^k(\vx_k^{p\star}, u_k^{p\star})$ is large, then the current cut-based lower approximation $\mathcal{V}_k^p$ is loose around $\vx_k^{p\star}$. We can add more cut-generation points around $\vx_k^{p\star}$ to refine the approximation for the $(p+1)$-th loop, which can accelerate convergence.
\end{remark}

\begin{remark}
    The global optimality of RDDP is guaranteed by an assumption that the backward reachable set $\mathcal{B}_k$ is exactly computed by \eqref{eq:recursion_backward} for each $k$. As will be introduced in Section \ref{subsec:ardp_copp3}, only a conservative approximation of $\mathcal{B}_k$ is available for COPP3 \cite{dio2025time}. As will be discussed in Section \ref{subsec:GOPP_KKT}, the global optimality of the original COPP problem \eqref{eq:op_discretize} can also be attained by iteratively updating reachable sets when conservative approximations of reachable sets are used.
\end{remark}

\section{RDDP for GOPP with KKT Convergence}\label{sec:RDDP_GOPP}

In Section \ref{sec:RDDP_COPP}, we have proposed RDDP, where global optimality of COPP is guaranteed by Theorem \ref{thm:adrp_optimal}. In practice, some objectives and constraints in OPP \eqref{eq:op_discretize} are non-convex, such as the jerk constraint \eqref{eq:jerk_constraint} in OPP3. However, the standard DDP is primarily designed for convex multistage problems, whereas cut-based approximations are not guaranteed to be global lower bounds of value functions without convexity.

A straightforward idea is to use RDDP as the convex-subproblem solver within a standard SCP framework. However, exact backward reachable sets of the convexified problems may be unavailable in OPP, whereas RDDP does not generally attain an exact optimizer of each subproblem within finitely many loops. Therefore, the standard exact-solve SCP analysis cannot be applied directly. Section \ref{subsec:gopp_process} presents an inexact DC decomposition process that considers conservative reachable sets and quasi-optimal subproblem solutions. Section \ref{subsec:GOPP_KKT} guarantees feasibility and KKT convergence under certain conditions. All proofs are given in Appendix \ref{app:proofs}.

\subsection{RDDP Process with Inexact DC Decomposition}\label{subsec:gopp_process}

Leveraging the global optimality of RDDP for COPP, this subsection applies DC decomposition to the non-convex functions in GOPP. According to the assumptions in Problem \ref{problem:general}, all functions $L_k$, $\Phi$, $\vf_k$, $\vh_k$, and the boundary sets $\mathcal{X}_0$, $\mathcal{X}_\text{f}$ in GOPP \eqref{eq:op_discretize} are DC-decomposable. For the $q$-th iteration of DC decomposition, $q\geq0$, the reference point $\bx_0^{(q)}=\big(\vx_k^{(q)}\big)_{k=0}^{N}$, $\bu_0^{(q)}=\big(u_k^{(q)}\big)_{k=0}^{N-1}$ is given. Through convexifying all non-convex functions and sets at $\big(\bx_0^{(q)},\bu_0^{(q)}\big)$, we construct the convex subproblem as follows:
\begin{subequations}\label{eq:op_discretize_convexity}
    \begin{align}
        \min_{\bx_0,\bu_0} \quad &\sum_{k=0}^{N-1} \overline{L}_k\big(\vx_k, u_k; \vx_k^{(q)},u_k^{(q)}\big) + \overline{\Phi}\big(\vx_N;\vx_N^{(q)}\big) \label{eq:op_discretize_convexity_objective}\\
        \st \quad & \vx_{k+1} = \vA_k\vx_k + \vb_k u_k,\label{eq:op_discretize_convexity_dynamics} \\
        &\overline{\vf}_k\big(\vx_k;\vx_k^{(q)}\big) \leq \vzero, \label{eq:op_discretize_convexity_state_constraints}\\
        &\overline{\vh}_k\big(\vx_k, u_k; \vx_k^{(q)}, u_k^{(q)}\big) \leq \vzero,\label{eq:op_discretize_convexity_mix_constraints}\\
        &\vx_0\in\overline{\mathcal{X}}_0\big(\vx_0^{(q)}\big), \,\,\vx_N\in\overline{\mathcal{X}}_\text{f}\big(\vx_N^{(q)}\big),\label{eq:op_discretize_convexity_boundary_constraints}
    \end{align}
\end{subequations}
where the convexified functions and sets are defined in Definition \ref{def:DC_decomposable}.
The convexified feasible set of \eqref{eq:op_discretize_convexity} is a subset of the original feasible set of GOPP \eqref{eq:op_discretize}; hence, any feasible solution of \eqref{eq:op_discretize_convexity} is also feasible in GOPP \eqref{eq:op_discretize}. In the ideal exact-solve DC process, an optimal solution of the convexified subproblem \eqref{eq:op_discretize_convexity} is used as the reference point for the next iteration, where solutions converge to a KKT point of GOPP \eqref{eq:op_discretize} under certain assumptions.

However, two issues hinder the direct application of this standard SCP framework to RDDP. First, exact reachable sets of the convexified problem \eqref{eq:op_discretize_convexity}, denoted by $\mathcal{B}_{\text{e}k}^{(q)}$, may be unavailable in practice. For example, backward reachable sets of COPP3 are typically approximated by a polytope with a limited number of edges \cite{dio2025time}, as detailed in Section \ref{subsec:ardp_copp3}. The conservative approximation of the backward reachable sets can be regarded as additional state constraints $\vr^{(q)}_k(\vx_k)\leq \vzero$ in the convexified subproblem \eqref{eq:op_discretize_convexity} for each iteration $q$. This tractable but conservative subproblem is as follows:
\begin{subequations}\label{eq:op_discretize_convexity_conservative}
    \begin{align}
        \min_{\bx_0,\bu_0} \quad &\sum_{k=0}^{N-1} \overline{L}_k\big(\vx_k, u_k; \vx_k^{(q)},u_k^{(q)}\big) + \overline{\Phi}\big(\vx_N;\vx_N^{(q)}\big)\label{eq:op_discretize_convexity_conservative_objective}\\
        \st \quad & \vr^{(q)}_k(\vx_k)\leq \vzero,\,\,k=0,\ldots,N-1,\label{eq:op_discretize_convexity_conservative_additional_constraints}\\
        & \text{Constraints \eqref{eq:op_discretize_convexity_dynamics}--\eqref{eq:op_discretize_convexity_boundary_constraints}}.\label{eq:op_discretize_convexity_conservative_other_constraints}
    \end{align}
\end{subequations}
Denote the conservative backward reachable set of problem \eqref{eq:op_discretize_convexity_conservative} by $\mathcal{B}_{\text{c}k}^{(q)}$.
The reachable sets of the three problems \eqref{eq:op_discretize}, \eqref{eq:op_discretize_convexity}, and \eqref{eq:op_discretize_convexity_conservative} satisfy
\[
\mathcal{B}_{\text{c}k}^{(q)}\subset\mathcal{B}_{\text{e}k}^{(q)}\subset\mathcal{B}_{k}.
\]
When exact reachable sets $\mathcal{B}_{\text{e}k}^{(q)}$ are available for all $k$, constraint \eqref{eq:op_discretize_convexity_conservative_additional_constraints} is redundant, and \eqref{eq:op_discretize_convexity_conservative} reduces to \eqref{eq:op_discretize_convexity}.

Second, after finitely many loops, Algorithm \ref{alg:ardp_copp} returns a feasible solution with a certified optimality gap by Theorem \ref{thm:near_optimality_gap}, rather than an exact optimal solution. These two sources of inexactness require a convergence analysis beyond the standard exact-solve DC process.

To address both issues, we consider an inexact DC process. Consider an arbitrary tolerance sequence $\{\varepsilon_q\}_{q=0}^\infty\subset(0,\infty)$ and an initial reference point $\big(\bx_0^{(0)},\bu_0^{(0)}\big)$. For each $q\geq0$, we convexify GOPP \eqref{eq:op_discretize} at $\big(\bx_0^{(q)},\bu_0^{(q)}\big)$ and solve the convexified conservative problem \eqref{eq:op_discretize_convexity_conservative} using RDDP. When constructing conservative reachable sets, we additionally require
\begin{equation}\label{eq:requirement_GOPP_feasible}
\vx_k^{(q)}\in\mathcal{B}_{\text{c}k}^{(q)},\,\forall q\geq1.
\end{equation}
The resulting solution $\big(\bx_0^{(q+1)},\bu_0^{(q+1)}\big)$ should satisfy
\begin{subequations}\label{eq:requirement_KKT_convergence}
    \begin{align}
    &\varepsilon_{\text{gap}}^{(q)}\big(\bx_0^{(q+1)},\bu_0^{(q+1)}\big)\leq\varepsilon_q,\,\forall q\geq0,\label{eq:requirement_KKT_convergence_epsilon_q}\\
    &J_{0}^{(q)}\big(\bx_0^{(q+1)},\bu_0^{(q+1)}\big)\leq J_{0}^{(q)}\big(\bx_0^{(q)},\bu_0^{(q)}\big),\,\forall q\geq1.\label{eq:requirement_KKT_convergence_Jq_decrease}
    \end{align}
\end{subequations}
Here, $\varepsilon_{\text{gap}}^{(q)}$ and $J_{0}^{(q)}$ denote the DDP optimality gap and the objective function \eqref{eq:op_discretize_convexity_conservative_objective} of the conservative problem \eqref{eq:op_discretize_convexity_conservative}, respectively. In particular, if $\big(\bx_0^{(q)},\bu_0^{(q)}\big)$ minimizes the convexified objective $J_{0}^{(q)}$ for some $q\geq1$, then we set $\bx_0^{(q+1)}=\bx_0^{(q)}$ and $\bu_0^{(q+1)}=\bu_0^{(q)}$. Finally, $\big(\bx_0^{(q+1)},\bu_0^{(q+1)}\big)$ serves as the reference point for the next iteration $q+1$.

\begin{algorithm}[!t]
    \caption{RDDP for GOPP (KKT convergence).}
    \label{alg:ardp_gopp}
    \begin{algorithmic}[1]
        \REQUIRE GOPP \eqref{eq:op_discretize}; a convergence criterion of DC decomposition; a tolerance sequence $\{\varepsilon_q\}_{q=0}^\infty$; an initial reference point $\big(\bx_0^{(0)}, \bu_0^{(0)}\big)$; inputs of Algorithm \ref{alg:ardp_copp}.
        \ENSURE A feasible trajectory $(\bx_0^\star,\bu_0^\star)$.
        \STATE Initialize $q\leftarrow 0$.
        \REPEAT
            \STATE Convexify GOPP \eqref{eq:op_discretize} at $\bx_0^{(q)},\bu_0^{(q)}$ to construct the subproblem \eqref{eq:op_discretize_convexity_conservative}, where the conservative reachable sets $\mathcal{B}_{\text{c}k}^{(q)}$ should satisfy local exactness \eqref{eq:KKT_approximate_reachable_set_condition} if $q\geq1$.
            \STATE Apply Algorithm \ref{alg:ardp_copp} to problem \eqref{eq:op_discretize_convexity_conservative}, and obtain a quasi-optimal solution $\big(\bx_0^{(q+1)}, \bu_0^{(q+1)}\big)$ satisfying \eqref{eq:requirement_KKT_convergence}.
            \STATE Update $q\leftarrow q+1$.
        \UNTIL{the convergence criterion is satisfied.}
        \STATE \textbf{return} the last solution $\big(\bx_0^{(q)}, \bu_0^{(q)}\big)$.
    \end{algorithmic}
\end{algorithm}

The inexact RDDP process for GOPP is summarized in Algorithm \ref{alg:ardp_gopp}. Here, $\varepsilon_{\text{gap}}^{(q)}$ certifies the optimality gap of the subproblem \eqref{eq:op_discretize_convexity_conservative}. For outer termination, we recommend a prescribed tolerance on the decrease in the real cost $J_0$ or the change in the reference point $\big(\bx_0^{(q)},\bu_0^{(q)}\big)$, together with a prescribed maximum number of outer iterations. As will be proved in Theorem \ref{thm:general_KKT}, any early terminated solution remains feasible in GOPP \eqref{eq:op_discretize}. Moreover, the time and space complexity of Algorithm \ref{alg:ardp_gopp} are $\mathcal{O}(N_{\text{outer}}T_{\text{COPP}})$ and $\mathcal{O}(mNN_\text{cut})$, respectively, where $N_{\text{outer}}$ is the number of outer DC iterations and $T_\text{COPP}$ denotes the time complexity of solving one COPP by RDDP. The computation of reachable sets and convexification is not included. It remains to analyze whether this iteration is well-defined, and how the conservative reachable sets $\mathcal{B}_{\text{c}k}^{(q)}$ and the inexact solutions with the optimality gaps $\varepsilon_q$ affect the KKT convergence.

\subsection{Feasibility and KKT Convergence Guarantees}\label{subsec:GOPP_KKT}

For simplicity, denote problems \eqref{eq:op_discretize}, \eqref{eq:op_discretize_convexity}, and \eqref{eq:op_discretize_convexity_conservative} by $\mathcal{P}$, $\mathcal{P}_{\text{e}}^{(q)}$, and $\mathcal{P}_{\text{c}}^{(q)}$, respectively. Their feasible sets are denoted by $\mathcal{F}$, $\mathcal{F}_{\text{e}}^{(q)}$, and $\mathcal{F}_{\text{c}}^{(q)}$, respectively. The following two definitions state the conditions for KKT convergence.

\begin{definition}\label{def:slater}
    Problem \eqref{eq:op_discretize_convexity} is said to satisfy the \textit{refined Slater condition} if it admits a feasible trajectory at which all inequality constraints hold strictly, except for the affine inequalities that are independent of the reference point.
\end{definition}

\begin{definition}\label{def:local_exactness}
    Define the one-step exact reachable set as
    \begin{align*}
        \widehat{\mathcal{B}}_{\text{e}k}^{(q)}=\big\{\vx_k\in\mathcal{B}_{\text{e}k}^{(q)}\mid&\exists u_k,\text{ s.t. } \overline{\vh}_k\big(\vx_k, u_k;\vx_k^{(q)},u_k^{(q)}\big) \leq \vzero, \notag\\
        &\vx_{k+1}=\vA_k\vx_k+\vb_k u_k\in\mathcal{B}_{\text{c}\,k+1}^{(q)}\big\}.
    \end{align*}
    The sequence $\big\{\mathcal{P}_{\text{c}}^{(q)}\big\}_{q=0}^\infty$ is said to satisfy \textit{local exactness} if $\exists\delta>0$, such that $\forall q\geq1$, $0\leq k<N$,
    \begin{equation}\label{eq:KKT_approximate_reachable_set_condition}
        \mathcal{B}_{\text{c}k}^{(q)}\cap\mathbb{B}\big(\vx_k^{(q)},\delta\big)=\widehat{\mathcal{B}}_{\text{e}k}^{(q)}\cap\mathbb{B}\big(\vx_k^{(q)},\delta\big).
    \end{equation}
    Here, $\mathbb{B}\big(\vx_k^{(q)},\delta\big)$ denotes the Euclidean ball centered at $\vx_k^{(q)}$ with radius $\delta$.
\end{definition}

Local exactness requires that the additional state constraints $\vr^{(q)}_k(\vx_k)\leq \vzero$ do not exclude feasible states of $\widehat{\mathcal{B}}_{\text{e}k}^{(q)}$ for each $k$ within a uniform neighborhood of $\vx_k^{(q)}$ across iterations. As will be discussed in Section \ref{sub2sec:OPP3_feasible_KKT}, this requirement could be difficult to enforce for curved boundaries of feasible regions. Note that in general SCP with certain conditions, local exactness is a sufficient condition for KKT convergence, rather than a general prerequisite \cite{bolte2016majorization}. Even without local exactness, the following theorem still guarantees feasibility and non-increasing real costs.

\begin{theorem}\label{thm:general_KKT}
    Consider Problem \ref{problem:general} and the RDDP process in Section \ref{subsec:gopp_process}. Assume that for each $q\geq0$, the conservative problem $\mathcal{P}_{\text{c}}^{(q)}$ satisfies the setting of Theorem \ref{thm:adrp_optimal}. Given an initial reference point $\big(\bx_0^{(0)},\bu_0^{(0)}\big)$ and tolerance sequence $\{\varepsilon_q\}_{q=0}^\infty\subset(0,\infty)$, assume that \eqref{eq:requirement_GOPP_feasible} holds and problem $\mathcal{P}_{\text{c}}^{(0)}$ is feasible, i.e., $\mathcal{F}_{\text{c}}^{(0)}\neq\varnothing$. The following conclusions hold.
    \begin{enumerate}
        \item\label{subthm:general_KKT_feasible} For each $q$, the conservative problem $\mathcal{P}_{\text{c}}^{(q)}$ is feasible, i.e., $\mathcal{F}_{\text{c}}^{(q)}\neq\varnothing$. Under the convergence criterion induced by \eqref{eq:requirement_KKT_convergence}, applying RDDP to $\mathcal{P}_{\text{c}}^{(q)}$ gives a feasible solution $\big(\bx_0^{(q+1)},\bu_0^{(q+1)}\big)\in\mathcal{F}_{\text{c}}^{(q)}\subset\mathcal{F}$ satisfying \eqref{eq:requirement_KKT_convergence}.
        \item\label{subthm:general_KKT_decresing} As $q$ increases, the real costs are non-increasing:
        \begin{equation}\label{eq:non_increasing_KKT}
            J_0\big(\bx_0^{(q+1)},\bu_0^{(q+1)}\big)\leq J_0\big(\bx_0^{(q)},\bu_0^{(q)}\big),\,\forall q\geq1.
        \end{equation}
        \item\label{subthm:general_KKT_convergence} Assume that $\lim_{q\to\infty}\varepsilon_q=0$ and $\big\{\mathcal{P}_{\text{c}}^{(q)}\big\}_{q=0}^\infty$ satisfies local exactness. Let $(\bx_0^\star,\bu_0^\star)$ be an accumulation point of $\big\{\big(\bx_0^{(q)},\bu_0^{(q)}\big)\big\}_{q=0}^\infty$, which exists due to the compactness of $\mathcal{F}$. If problem \eqref{eq:op_discretize_convexity} convexified at $(\bx_0^\star,\bu_0^\star)$, denoted by $\mathcal{P}_{\text{e}}^\star$, satisfies the refined Slater condition, then $(\bx_0^\star,\bu_0^\star)$ is a KKT point of the original GOPP \eqref{eq:op_discretize}.
    \end{enumerate}
\end{theorem}

Theorem \ref{thm:general_KKT} establishes well-definedness, feasibility preservation, and conditional KKT convergence of the process in Section \ref{subsec:gopp_process}. The initial reference $\big(\bx_0^{(0)},\bu_0^{(0)}\big)$ need not be feasible in GOPP; only the first conservative subproblem $\mathcal{P}_{\text{c}}^{(0)}$ must be feasible. For example, an OPP2 solution is usually used to construct an initial reference for OPP3 \cite{wang2026online,wang2026onlineSLP}, although it violates 3rd-order constraints. Moreover, exact optimal solutions for subproblems are not required, while the KKT convergence additionally requires $\lim_{q\to\infty}\varepsilon_q=0$.

\begin{remark}
    We now discuss some special cases. First, if exact backward reachable sets $\mathcal{B}_{\text{e}k}^{(q)}$ are available, then problem \eqref{eq:op_discretize_convexity_conservative} reduces to \eqref{eq:op_discretize_convexity}, and local exactness \eqref{eq:KKT_approximate_reachable_set_condition} holds automatically. Second, the ideal case with available optimal solution of every conservative problem \eqref{eq:op_discretize_convexity_conservative} can be viewed as $\varepsilon_q\equiv0$. Third, if the original GOPP \eqref{eq:op_discretize} is convex with conservative reachable sets, then KKT convergence is strengthened: any accumulation point $(\bx_0^\star,\bu_0^\star)$ is globally optimal in problem \eqref{eq:op_discretize}.
\end{remark}

\section{RDDP for OPP2 and OPP3}\label{sec:ardp_opp2_opp3}

This section develops efficient instantiations of RDDP for OPP2 and OPP3. Specifically, this section focuses on COPP since GOPP can be solved by the process in Section \ref{sec:RDDP_GOPP}.

\subsection{RDDP for COPP2}\label{subsec:ardp_copp2}

For COPP2, the backward reachable sets can be computed exactly as intervals by existing RA techniques. The one-dimensional state is denoted by $\vx_k=x_k:= x_{k,1}=\frac12\dot{s}^2(s_k)\in\R_+$. Then, the dynamic equation \eqref{eq:op_discretize_dynamics} implies that $x_{k+1}=x_k+u_k\Delta_k$.

\subsubsection{Bidirectional Reachable Set}\label{sub2sec:ardp_opp2_reachable_set}
We can directly apply the RA approach in \cite{pham2018new} for COPP2. To improve the sampling efficiency of RDDP, we choose to compute the bidirectional reachable set instead of the backward reachable set. Note that such a modification does not influence the optimality and convergence of RDDP, as the bidirectional reachable set is the subset of the backward reachable set of those states that can be reached from the initial state set $\mathcal{X}_0$. The backward reachable set $\mathcal{B}_{k}'=[\underbar{x}_{k}',\bar{x}_{k}']$ is computed via BP in \eqref{eq:recursion_backward}, as shown in Fig. \ref{fig:reachable_set_copp2}(a). Then, as shown in Fig. \ref{fig:reachable_set_copp2}(b), the bidirectional reachable set $\mathcal{B}_k=[\underbar{x}_k,\bar{x}_k]$ is computed via FP by solving the following two-dimensional convex optimization problem:
    \begin{align*}
        \minmax_{x_k,u_{k-1}}\quad&x_k\\
        \st\quad&\vh_{k-1}(x_k-u_{k-1}\Delta_{k-1},u_{k-1})\leq\vzero,\\
        &\underbar{x}_{k-1}\leq x_k-u_{k-1}\Delta_{k-1}\leq \bar{x}_{k-1},\\
        &\underbar{x}_k'\leq x_k\leq \bar{x}_k',
    \end{align*}
where $\mathcal{B}_{0}=\mathcal{B}_{0}'\cap\mathcal{X}_0$. Let $\mathcal{O}(N_\text{cons})$ bound the number of scalar state and mixed constraints at each stage, and let $\mathcal{O}(T_\text{2D})$ uniformly bound the cost of solving each two-dimensional convex subproblem to the prescribed accuracy. Computing all bidirectional reachable sets requires $\mathcal{O}(NT_\text{2D})$ time and $\mathcal{O}(N)$ storage. For linear constraints, Seidel's randomized algorithm gives an expected cost of $T_\text{2D}=\mathcal{O}(N_\text{cons})$ per subproblem \cite[Theorem 1]{seidel1991small}, yielding an expected total time of $\mathcal{O}(NN_\text{cons})$.

\begin{figure}[!]
    \centering
    \input{figures/reachable_set_copp2.tikz}
    \caption{Computation of reachable sets for COPP2. (a) Backward propagation of the backward reachable set $\mathcal{B}_k'$. (b) Forward propagation of the bidirectional reachable set $\mathcal{B}_k$.}
    \label{fig:reachable_set_copp2}
\end{figure}

\begin{figure*}[!]
    \centering
    \input{figures/reachable_set_copp3_extend.tikz}
    \caption{Computation of backward reachable sets for COPP3 based on projection. (a) The one-step exact reachable set $\widehat{\mathcal{B}}_{\text{e}k}$ with sampled vertices. (b) The proposed extension-based method in Section \ref{sub2sec:backward_set_copp3_extend}. (c) The baseline in \cite{dio2025time}. The proposed method provides a more accurate estimation of the residual area.}
    \label{fig:reachable_set_copp3_extend}
\end{figure*}

\subsubsection{Backward Cut-Generation Pass}\label{sub2sec:ardp_opp2_subgradient}

The backward cut-generation pass of RDDP for COPP2 follows the standard procedure in Section \ref{sub2sec:backward_cut_pass}. The admissible control set $\mathcal{C}_k(x_k)=[\underbar{u}_k(x_k),\bar{u}_k(x_k)]$ can be solved by the one-dimensional convex optimization problem \eqref{eq:admissable_control_set}, where the active constraints can be used to compute the directional derivatives $\frac{\mathrm{d}\underbar{u}_k}{\mathrm{d}x_k}(x_k)$ and $\frac{\mathrm{d}\bar{u}_k}{\mathrm{d}x_k}(x_k)$. Then, the cut-based approximation $\mathcal{V}_k^p$ is constructed by the one-dimensional convex problem \eqref{eq:approximate_value_function_backward}.

\subsubsection{Forward Trial Pass} In the forward pass, approximate value functions $\mathcal{V}_k^p$ have been computed for all $k$. Following the standard process in Section \ref{sub2sec:refinement}, the trial trajectory is generated by solving a sequence of one-dimensional convex optimization problems \eqref{eq:forward_refine}.

Since $\mathcal{B}_k$ is exactly computed for COPP2, the global optimality convergence of the trial trajectories with respect to (w.r.t.) the original COPP \eqref{eq:op_discretize} is guaranteed by Theorem \ref{thm:adrp_optimal}.

\subsection{RDDP for COPP3}\label{subsec:ardp_copp3}

In this subsection, COPP3 typically denotes the convexified subproblem of GOPP3 with non-convex 3rd-order constraints. The backward cut-generation pass and forward trial pass follow Section \ref{subsec:rddp_process}. We therefore focus on a challenge for RA-based COPP3: the computation of backward reachable sets, which are convex but generally intractable. According to \eqref{eq:backward_reachable_definition}, as $k$ decreases, the growing number of constraints for representing the exact reachable set $\mathcal{B}_{\text{e}k}$, including active boundary constraints, makes exact representation intractable.

Leveraging recent work \cite{dio2025time} and our established Theorem \ref{thm:general_KKT}, we can compute the backward reachable sets by a conservative polytope with a limited number of edges, while guaranteeing KKT convergence for GOPP3 under local exactness conditions. Specifically, this subsection presents two methods for computing conservative backward reachable sets for efficiency under different conditions. The improved reachable-set computation can also be plugged into TOPP3-RA solvers, leading to the TOPP3-RA-IRC variant evaluated in this paper.

In contrast to COPP2, we recommend using backward reachable sets instead of bidirectional reachable sets in COPP3, as the approximation of bidirectional reachable sets would increase the conservativeness and hinder the optimality. In COPP3, the state is denoted by $\vx_k=(x_k,y_k):=(\frac12\dot{s}^2(s_k),\ddot{s}(s_k))\in\R_+\times\R$. Then, \eqref{eq:op_discretize_dynamics} implies $x_{k+1}=x_k+y_k\Delta_k+\frac12u_k\Delta_k^2$ and $y_{k+1}=y_k+u_k\Delta_k$. In this subsection, the exact, one-step exact, and conservative backward reachable sets are denoted by $\mathcal{B}_{\text{e}k}$, $\widehat{\mathcal{B}}_{\text{e}k}$, and $\mathcal{B}_{\text{c}k}$, respectively.

In this subsection, the exact, one-step exact, and conservative backward reachable sets are denoted by $\mathcal{B}_{\text{e}k}$, $\widehat{\mathcal{B}}_{\text{e}k}$, and $\mathcal{B}_{\text{c}k}$, respectively, consistently with Definition \ref{def:local_exactness}.

\subsubsection{Projection-Based Backward Reachable Set}\label{sub2sec:backward_set_copp3_extend}

\begin{figure*}[!]
    \centering
    \input{figures/reachable_set_copp3_clip.tikz}
    \caption{Computation of the one-step exact backward reachable set $\widehat{\mathcal{B}}_{\text{e}k}$ for COPP3 based on searching. (a) The geometric projection relation between $\widehat{\mathcal{B}}_{\text{e}k}$ and $\widetilde{\mathcal{B}}_k$. (b) Searching along a line. (c) Searching along a plane.}
    \label{fig:reachable_set_copp3_clip}
\end{figure*}

We establish an improved method of \cite[Section \RomanNum{3}-B2]{dio2025time}. Let $\mathcal{B}_{\text{c}N}=\mathcal{X}_\text{f}$. Given the conservative convex set $\mathcal{B}_{\text{c}\,k+1}$, the one-step exact backward reachable set $\widehat{\mathcal{B}}_{\text{e}k}$ is the projection of
\begin{align}
    \widetilde{\mathcal{B}}_k=&\Big\{(x_k, y_k, u_k)\mid \vf_k(x_k,y_k)\leq\vzero,\,\vh_k(x_k,y_k,u_k)\leq\vzero, \notag\\
    &\Big(x_k+y_k\Delta_k+\frac12u_k\Delta_k^2,\, y_k+u_k\Delta_k\Big)\in\mathcal{B}_{\text{c}\,k+1}\Big\}.\label{eq:reachable_set_copp3_3d_before_projection}
\end{align}
To project a boundary point of $\mathcal{B}_{\text{e}k}$ to a vertex of the conservative reachable set $\mathcal{B}_{\text{c}k}$, we can solve the following three-dimensional convex optimization problem:
\begin{equation}\label{eq:reachable_set_copp3_extend}
    (x_k^i, y_k^i, u_k^i)\in\argmax_{(x_k,y_k,u_k)\in\widetilde{\mathcal{B}}_k} a_{i}x_k+ b_{i}y_k,
\end{equation}
where $(a_i,b_i)\not=\vzero$ is the sampling direction.

As shown in Fig. \ref{fig:reachable_set_copp3_extend}(a), the one-step exact $\widehat{\mathcal{B}}_{\text{e}k}$ can be obtained by sampling infinitely many directions, which is intractable. In practice, an upper bound of the number of vertices $P_k$ is given. Then, we can sample no more than $P_k$ vertices of $\widehat{\mathcal{B}}_{\text{e}k}$ with different sampling directions $(a_i,b_i)$. As a subset of $\widehat{\mathcal{B}}_{\text{e}k}$, the convex hull of the sampled vertices is computed as the conservative backward reachable set $\mathcal{B}_{\text{c}k}$.

It remains to choose the sampling directions. Inspired by \cite{dio2025time}, we heuristically maximize the area of $\mathcal{B}_{\text{c}k}$ to reduce conservativeness. Assume that $i<P_k$ vertices have been sampled, denoted by $\{\vx_k^j=(x_k^j, y_k^j)\}_{j=1}^i$. As shown in Fig. \ref{fig:reachable_set_copp3_extend}(b1), their convex hull, denoted by $\mathcal{B}_\text{in}$, is an inner approximation of $\widehat{\mathcal{B}}_{\text{e}k}$. The left and right tangent lines of $\widehat{\mathcal{B}}_{\text{e}k}$ at each $\vx_k^j$ can be computed by active constraints of \eqref{eq:reachable_set_copp3_extend}. As shown in Fig. \ref{fig:reachable_set_copp3_extend}(b2), denote by $\mathcal{B}_{\text{out}}$ the outer approximation of $\widehat{\mathcal{B}}_{\text{e}k}$ enclosed by these tangent lines. Hence,
\begin{equation*}
    \mathcal{B}_\text{in}\subseteq\widehat{\mathcal{B}}_{\text{e}k}\subseteq\mathcal{B}_\text{out}.
\end{equation*}
Although $\widehat{\mathcal{B}}_{\text{e}k}$ is intractable, $\mathcal{B}_\text{in}$ and $\mathcal{B}_\text{out}$ are both tractable. Denote the vertices of $\mathcal{B}_\text{out}$ by $\{\hat{\vx}_k^j\}_{j=1}^i$. Then, $\Setminus{\mathcal{B}_\text{out}}{\mathcal{B}_\text{in}}$ is the union of triangles formed by $\hat{\vx}_k^j$, $\vx_k^j$, and $\vx_k^{j+1}$, where $\vx_k^{i+1}$ refers to $\vx_k^1$. Similar to \cite{dio2025time}, we select the triangle with the largest area. Then, the next vertex is sampled with the outer normal direction of the corresponding segment $\vx_k^j$--$\vx_k^{j+1}$. The above process is repeated until $P_k$ different vertices are sampled or $\widehat{\mathcal{B}}_{\text{e}k}$ is exactly obtained. Finally, set $\mathcal{B}_{\text{c}k}$ to the convex hull of all sampled vertices.

\begin{remark}
    The difference between the proposed method and \cite{dio2025time} lies in the construction of the outer polytope $\mathcal{B}_\text{out}$. In \cite{dio2025time}, the outer polytope is constructed by the support lines of \eqref{eq:reachable_set_copp3_extend}, i.e., the orthogonal lines to $(a_i,b_i)$, as shown in Fig. \ref{fig:reachable_set_copp3_extend}(c). In contrast, our method constructs $\mathcal{B}_\text{out}$ by the two-sided tangent lines at each sampled vertex, which provides a more accurate estimation of the residual area. As shown in Fig. \ref{fig:reachable_set_copp3_extend}(b)-(c), the proposed method empirically achieves a tighter approximation of $\widehat{\mathcal{B}}_{\text{e}k}$ than \cite{dio2025time}, which contributes to a better optimality of the generated trajectory.
\end{remark}

Let $N_\text{ver}$ denote the maximum number of vertices in conservative reachable polygons. Every stage requires $\mathcal{O}(N_\text{ver})$ three-dimensional convex subproblems, each involving $\mathcal{O}(N_\text{cons}+N_\text{ver})$ constraints. Let $\mathcal{O}(T_\text{3D})$ uniformly bound the cost of each solve, satisfying $\mathcal{O}(T_\text{3D})=\mathcal{O}(N_\text{cons}+N_\text{ver})$ in Seidel's randomized algorithm for linear programming (LP) \cite[Theorem 1]{seidel1991small}. Using a max-heap to select the next optimization direction incurs a total cost of $\mathcal{O}(N_\text{ver}\log N_\text{ver})$ per stage. Thus, the expected time and space complexity for reachable sets are $\mathcal{O}(NN_\text{ver}(T_\text{3D}+\log N_\text{ver}))$ and $\mathcal{O}(NN_\text{ver})$, respectively.

\subsubsection{Searching-Based Backward Reachable Set}\label{sub2sec:backward_set_copp3_clip}

This section proposes a more efficient method for computing the backward reachable set for the case where all constraints in COPP3 \eqref{eq:op_discretize} are linear. Note that the linearity of constraints is a common case in practice \cite[Appendix A.2]{wang2026online}. By Lemma \ref{lemma:convexity_reachset_valfunc}, the exact backward reachable set is a convex polytope in $\R^2$. As shown in Fig. \ref{fig:reachable_set_copp3_clip}(a), for a given polytopic $\mathcal{B}_{\text{c}\,k+1}$, $\widehat{\mathcal{B}}_{\text{e}k}$ is the projection of the three-dimensional convex polytope $\widetilde{\mathcal{B}}_k$ in \eqref{eq:reachable_set_copp3_3d_before_projection}.

The key idea is to directly compute the one-step exact $\widehat{\mathcal{B}}_{\text{e}k}$ by iterative vertex searching. Specifically, the first vertex $\vx_k^1$ of $\widehat{\mathcal{B}}_{\text{e}k}$ is obtained through the three-dimensional LP \eqref{eq:reachable_set_copp3_extend} with an arbitrary sampling direction. Assume that the last sampled three-dimensional vertex is $(x_k^i, y_k^i, u_k^i)$ with left and right tangent directions $\vd_{i}^-$ and $\vd_i^+$. Then, we can search for the next vertex along the right tangent direction $\vd_i^+$ which corresponds to either an edge or a plane of $\widetilde{\mathcal{B}}_k$. As shown in Fig. \ref{fig:reachable_set_copp3_clip}, a one- or two-dimensional LP is solved to obtain the next vertex for the two cases, respectively.

Finally, the standard ear-clipping method is applied to the one-step exact $\widehat{\mathcal{B}}_{\text{e}k}$ to reduce the number of vertices to be no more than $P_k$ for efficiency. Specifically, the ear-clipping method iteratively removes the triangle of the smallest area formed by three adjacent vertices until the number of vertices is no more than $P_k$. The resulting convex hull of the remaining vertices is the conservative backward reachable set $\mathcal{B}_{\text{c}k}$.

\begin{remark}
    Intuitively, the projection-based method in Section \ref{sub2sec:backward_set_copp3_extend} can save computation when $P_k$ is small, as the searching-based method fully solve all vertices of $\mathcal{B}_{\text{e}k}$. In practice, however, searching is substantially more efficient, as it requires one- or two-dimensional LP for each vertex, rather than three-dimensional convex optimization. Therefore, we recommend the searching- and projection-based methods for linear and non-linear COPP3, respectively.
\end{remark}

At stage $k$, $\widetilde{\mathcal{B}}_k$ is defined by $\mathcal{O}(N_\text{cons}+N_\text{ver})$ linear constraints. It has $\mathcal{O}(N_\text{cons}+N_\text{ver})$ vertices \cite[Section 10.1]{grunbaum2003convex}; hence, its projection $\widehat{\mathcal{B}}_{\text{e}k}$ also has $\mathcal{O}(N_\text{cons}+N_\text{ver})$ vertices. Each LP can be solved in expected $\mathcal{O}(N_\text{cons}+N_\text{ver})$ time using Seidel's randomized algorithm \cite[Theorem~1]{seidel1991small}. Hence, the time and space complexity for conservative reachable sets, together with ear clipping, are $\mathcal{O}\left(N(N_\text{cons}+N_\text{ver})^2\right)$ and $\mathcal{O}(NN_\text{ver}+N_\text{cons})$, respectively.

\subsubsection{Guarantee of Feasibility and KKT Convergence}\label{sub2sec:OPP3_feasible_KKT}

We now examine \eqref{eq:requirement_GOPP_feasible} and the local exactness condition \eqref{eq:KKT_approximate_reachable_set_condition} in Theorem \ref{thm:general_KKT} for the above two methods.

For the projection-based method in Section \ref{sub2sec:backward_set_copp3_extend}, \eqref{eq:requirement_GOPP_feasible} can be enforced by merging $\vx_{k}^{(q)}$ into $\mathcal{B}_{\text{c}k}$ by computing their convex hull, which remains a polytope. Therefore, feasibility and non-increasing real costs are guaranteed. One can also extend $\mathcal{B}_{\text{c}k}$ around $\vx_{k}^{(q)}$ along various feasible directions. However, whenever the one-step exact reachable set has a curved boundary which lies in $\mathbb{B}\big(\vx_k^{(q)},\delta\big)$, local exactness \eqref{eq:KKT_approximate_reachable_set_condition} cannot be enforced within finitely many vertices. So Theorem \ref{thm:general_KKT} alone fails to guarantee KKT convergence when involving such curved boundaries.

For the searching-based method in Section \ref{sub2sec:backward_set_copp3_clip}, \eqref{eq:requirement_GOPP_feasible} and local exactness \eqref{eq:KKT_approximate_reachable_set_condition} can be easily guaranteed. Specifically, when applying the ear-clipping method, triangles that intersect with $\mathbb{B}\big(\vx_k^{(q)},\delta\big)$ are retained and not clipped. In this way, KKT convergence is guaranteed under the remaining assumptions of Theorem \ref{thm:general_KKT}.

\section{Experiments}

\subsection{Setup}\label{subsec:setup}

The proposed RDDP and all baselines are implemented in Rust. Experiments are conducted on a laptop with an Intel${}^\text{\textregistered}$ Core${}^\text{TM}$ Ultra 9 Processor 275HX. Unless otherwise specified, optimization problems in baselines are solved using Clarabel \cite{goulart2026clarabel}. To assess the influence of solver choice, Section \ref{sub2sec:random_7axis_path} also evaluates the SOCP baselines using QOCO \cite{chari2026qoco}, ECOS \cite{domahidi2013ecos}, and Gurobi \cite{gurobi}. The tested methods are as follows:
\begin{itemize}
    \item \textbf{TOPP2-RA} / \textbf{TOPP3-RA}: RA methods for TOPP2 \cite{pham2018new} or TOPP3 \cite{dio2025time}. All GOPP3 problems are convexified at the solution $x_1(s)$ of TOPP2-RA. In TOPP3-RA, the maximal number of vertices for the conservative backward reachable set is set to 16.
    \item \textbf{TOPP3-RA-IRC} (Ours): In contrast to TOPP3-RA, we apply the proposed improved reachable set (IRC) computation method in Section \ref{subsec:ardp_copp3} to TOPP3-RA-IRC, which should be viewed as our contribution.
    \item \textbf{GOPP2-SOCP} / \textbf{GOPP3-SOCP}: The well-known CO methods for GOPP2 \cite{verscheure2009time} or GOPP3 \cite{debrouwere2013time} formulated as second-order cone programs (SOCP). For convex OPP2, the globally optimal value is solved. For non-convex OPP3, SOCP is embedded into SCP, which converges to a KKT point without guaranteeing global optimality.
    \item \textbf{GOPP2-SOCP${}^*$} / \textbf{GOPP3-SOCP${}^*$}: SOCP variants with a time limit per optimization problem: 50\,ms for GOPP2 and 150\,ms for GOPP3 in Section \ref{sub2sec:random_7axis_path}. The historical best solution is returned. Clarabel usually overruns the time limit slightly.
    \item \textbf{GOPP2-IDP} / \textbf{GOPP3-IDP}: Grid-based ideal DP \cite{shin1986dynamic}. At each $s_k$, the state grid is discretized with 60 and 30$\times$30 points in GOPP2 and GOPP3, respectively.
    \item \textbf{GOPP3-GDP}: DP with a greedy beam-search heuristic for GOPP3 \cite{kaserer2019nearly}. The number of control segments is set to 100. For each control point, the number of sampled states is set to 25.
    \item \textbf{GOPP2-RDDP} / \textbf{GOPP3-RDDP} (Ours): At each $s_k$, RDDP starts with 5 and 5$\times$5 cut-generation states in GOPP2 and GOPP3, respectively. In Sections \ref{subsec:numerical_experiments} and \ref{subsec:real_world_experiment}, GOPP2 and GOPP3 use 3 and 5 RDDP loops, respectively. Reachable sets of OPP3 are computed as in Section \ref{subsec:ardp_copp3} with the same vertex limit as TOPP3-RA.
\end{itemize}

Section \ref{subsec:toy_example} examines the theoretical convergence behavior of RDDP. Section \ref{subsec:numerical_experiments} evaluates the balance between solution quality and computation time, while Section \ref{subsec:real_world_experiment} demonstrates the practical use of general-objective trajectories on a robot. For Sections \ref{subsec:numerical_experiments} and \ref{subsec:real_world_experiment}, the initial sample counts and loop budgets of RDDP specified above are chosen to efficiently generate high-quality feasible trajectories.

\begin{figure}[!t]
    \centering
    \input{figures/copp2_case_study_timing.tikz}
    \caption{Influence of initial sample numbers on RDDP in COPP2. (a) DDP optimality gap $\varepsilon_{\text{gap}}$. (b) Real cost $J_0$. (c)~Mean cumulative computation time $T_\text{c}$ over 100 runs.}
    \label{fig:case_copp2}
\end{figure}

\subsection{Numerical Case Study}\label{subsec:toy_example}

This subsection studies a numerical case to examine the convergence theory of RDDP methods. A fixed path for a 7-axis Franka Emika Panda manipulator is considered and discretized into 1000 $s$-intervals. Joint torques are evaluated using the calibrated Panda dynamics in \eqref{eq:robot_dynamics}, with asymmetric Coulomb friction and without viscous friction, as in Section \ref{subsec:real_world_experiment}. The joint constraints are set as follows: the maximal velocity $\dot{\vq}^+$ is [1.1, 1.1, 2, 1.1, 1.3, 1.3, 1.3] rad/s; the maximal acceleration $\ddot{\vq}^+$ is [6, 6, 8, 10, 12, 16, 16] rad/s${}^\text{2}$; the maximal jerk $\dddot{\vq}^+$ is [350, 490, 560, 560, 560, 560, 560] rad/s${}^\text{3}$; the maximal torque $\vtau^+$ is [52.2, 52.2, 52.2, 52.2, 7.2, 7.2, 7.2] N$\cdot$m. The objective is a hybrid of traversal time and thermal energy, $J_{\text{time}}+5J_{\text{thermal}}=\int_{0}^{t_\f}\left(1+5\norm[2]{\frac{\vtau(t)}{\vtau^+}}^2\right)\,\mathrm{d}t$.

For COPP2, the convergence criterion is set to $\varepsilon_\text{gap}\leq10^{-3}$. As shown in Fig. \ref{fig:case_copp2}, the initial number of cut-generation samples affects the initial optimality gap, but not convergence to the same optimal value here. Starting with 3 rather than 17 samples requires only 3 additional loops to reach this tolerance. This suggests that RDDP is more sample-efficient than uniform sampling, which helps explain its faster convergence compared with grid-based IDP in Section \ref{sub2sec:random_7axis_path}. Fig. \ref{fig:case_copp2}(c) shows that total computation time varies non-monotonically with the initial sample count: more samples increase per-loop computation time, whereas fewer initial samples may require more loops to reach the prescribed tolerance.

\begin{figure}[!t]
    \centering
    \input{figures/gopp3_sampling_case_study_timing.tikz}
    \caption{Influence of initial sample numbers on RDDP in GOPP3. (a) DDP optimality gap $\varepsilon_{\text{gap}}^{(q)}$ for conservative convexified subproblems. (b) Real cost $J_0$. (c) Mean cumulative computation time $T_\text{c}$ over 100 runs.}
    \label{fig:case_gopp3_sampling}
\end{figure}

For GOPP3, the convergence criterion is set to $\varepsilon_{\text{gap}}^{(q)}\leq\varepsilon_q:=10^{-q-1}$ as in \eqref{eq:requirement_KKT_convergence_epsilon_q} for $q=0,1,2$. As shown in Fig. \ref{fig:case_gopp3_sampling}, the DDP optimality gap $\varepsilon_{\text{gap}}^{(q)}$ is reduced to within the prescribed tolerance in each convex subproblem, but typically increases temporarily when the next subproblem is formed. The real costs of the converged subproblem solutions decrease monotonically, consistent with Section \ref{sec:RDDP_GOPP}. Fig. \ref{fig:case_gopp3_sampling}(c) shows that an excessively large initial sample count can substantially increase computation time.

As shown in Fig. \ref{fig:case_gopp3_polygon}, too few vertices slow convergence, whereas further increases have little effect once sufficiently many vertices are used. For the vertex number is constrained no more than 6, we allow one additional outer iteration with an inner gap tolerance of $\varepsilon_3=10^{-3}$. Preserving local exactness yields a real cost closer to the SOCP KKT reference than omitting it, highlighting the importance of local-exactness \eqref{eq:KKT_approximate_reachable_set_condition}. Fig. \ref{fig:case_gopp3_polygon}(c) shows that per-loop computation time varies little with the vertex number under the current settings; hence, reducing the number of loops needed for convergence is more effective than reducing the vertex number.

\begin{figure}[!t]
    \centering
    \input{figures/gopp3_polygon_case_study_timing.tikz}
    \caption{Influence of reachable sets' vertex numbers and local exactness on RDDP in GOPP3. (a) DDP optimality gap $\varepsilon_{\text{gap}}^{(q)}$ for conservative convexified subproblems. (b) Real cost $J_0$. (c)~Mean cumulative computation time $T_\text{c}$ over 100 runs.}
    \label{fig:case_gopp3_polygon}
\end{figure}

\begin{table*}[!t]
    \centering
    \caption{Performance comparison of different methods for TOPP2, GOPP2, TOPP3, and GOPP3 (mean $\pm$ std). For SOCP baselines, C, Q, E, and G in parentheses denote the solvers Clarabel, QOCO, ECOS, and Gurobi, respectively.}
    \label{tab:performance_comparison}
    \footnotesize
    \setlength{\tabcolsep}{2pt}
    \renewcommand{\arraystretch}{1}
    \subcaptionbox{Time-Optimal Path Parameterization (TOPP)\label{tab:performance_comparison_topp}}[0.485\textwidth]{%
    \begin{tabular}{lllll}
        \toprule
        Problem & Method & SR (\%) & $T_\text{c}$ (ms) & $J_1$ \\
        \midrule
        \multirow{8}{*}{TOPP2} & TOPP2-RA & 100 & 0.433$\pm$0.031 & 30.399$\pm$1.887 \\
        & GOPP2-SOCP (C) & 100 & 156.582$\pm$10.010 & 30.399$\pm$1.887 \\
        & GOPP2-SOCP (Q) & 100 & 92.040$\pm$6.280 & 30.399$\pm$1.887 \\
        & GOPP2-SOCP (E) & 100 & 151.547$\pm$9.851 & 30.399$\pm$1.887 \\
        & GOPP2-SOCP (G) & 100 & 340.726$\pm$79.690 & 30.399$\pm$1.887 \\
        & GOPP2-SOCP${}^*$ (C) & 100 & 61.616$\pm$1.550 & 30.649$\pm$1.892 \\
        & GOPP2-IDP & 100 & 13.029$\pm$0.362 & 30.412$\pm$1.886 \\
        & \textbf{GOPP2-RDDP} & \textbf{100} & \textbf{3.716$\pm$0.247} & \textbf{30.399$\pm$1.887} \\ \midrule
        \multirow{10}{*}{TOPP3} & TOPP3-RA & 100 & 24.299$\pm$1.401 & 30.768$\pm$1.886 \\
        & \textbf{TOPP3-RA-IRC} & 100 & 9.742$\pm$0.444 & 30.768$\pm$1.885 \\
        & GOPP3-SOCP (C) & 100 & 345.298$\pm$22.141 & 30.760$\pm$1.886 \\
        & GOPP3-SOCP (Q) & 100 & 492.245$\pm$223.944 & 30.760$\pm$1.886 \\
        & GOPP3-SOCP (E) & 100 & 316.758$\pm$23.527 & 30.760$\pm$1.886 \\
        & GOPP3-SOCP (G) & 100 & 2129.893$\pm$966.061 & 30.760$\pm$1.886 \\
        & GOPP3-SOCP${}^*$ (C) & 100 & 169.713$\pm$2.420 & 30.812$\pm$1.908 \\
        & GOPP3-IDP & 100 & 12262.139$\pm$850.531 & 32.003$\pm$1.766 \\
        & GOPP3-GDP & 62 & 69.471$\pm$3.349 & 32.699$\pm$1.773 \\
        & \textbf{GOPP3-RDDP} & \textbf{100} & \textbf{64.503$\pm$4.099} & \textbf{30.761$\pm$1.886} \\ \bottomrule
    \end{tabular}%
    }\hfill
    \subcaptionbox{General Optimal Path Parameterization (GOPP)\label{tab:performance_comparison_gopp}}[0.485\textwidth]{%
    \begin{tabular}{lllll}
        \toprule
        Problem & Method & SR (\%) & $T_\text{c}$ (ms) & $J_2$ \\
        \midrule
        \multirow{8}{*}{GOPP2} & TOPP2-RA & 100 & 0.431$\pm$0.037 & 58.239$\pm$2.531 \\
        & GOPP2-SOCP (C) & 100 & 267.751$\pm$25.970 & 46.191$\pm$1.572 \\
        & GOPP2-SOCP (Q) & 100 & 142.995$\pm$11.057 & 46.191$\pm$1.572 \\
        & GOPP2-SOCP (E) & 100 & 257.289$\pm$29.873 & 46.191$\pm$1.572 \\
        & GOPP2-SOCP (G) & 100 & 864.126$\pm$219.427 & 46.191$\pm$1.572 \\
        & GOPP2-SOCP${}^*$ (C) & 100 & 68.429$\pm$2.327 & 46.911$\pm$3.057 \\
        & GOPP2-IDP & 100 & 14.292$\pm$0.525 & 46.261$\pm$1.566 \\
        & \textbf{GOPP2-RDDP} & \textbf{100} & \textbf{9.370$\pm$0.334} & \textbf{46.205$\pm$1.570} \\
        \midrule
        \multirow{10}{*}{GOPP3} & TOPP3-RA & 100 & 24.610$\pm$3.795 & 57.538$\pm$2.416 \\
        & \textbf{TOPP3-RA-IRC} & 100 & 9.813$\pm$0.641 & 57.529$\pm$2.418 \\
        & GOPP3-SOCP (C) & 100 & 356.834$\pm$42.674 & 46.403$\pm$1.573 \\
        & GOPP3-SOCP (Q) & 99 & 534.023$\pm$190.360 & 46.390$\pm$1.576 \\
        & GOPP3-SOCP (E) & 100 & 304.478$\pm$29.343 & 46.403$\pm$1.573 \\
        & GOPP3-SOCP (G) & 100 & 3878.345$\pm$1212.820 & 46.403$\pm$1.573 \\
        & GOPP3-SOCP${}^*$ (C) & 100 & 175.139$\pm$3.258 & 48.562$\pm$4.006 \\
        & GOPP3-IDP & 100 & 11617.767$\pm$675.389 & 59.994$\pm$3.470 \\
        & GOPP3-GDP & 68 & 70.775$\pm$2.179 & 47.151$\pm$1.459 \\
        & \textbf{GOPP3-RDDP} & \textbf{100} & \textbf{61.935$\pm$0.857} & \textbf{47.131$\pm$1.550} \\ \bottomrule
    \end{tabular}%
    }
\end{table*}

\subsection{Numerical Benchmarks}\label{subsec:numerical_experiments}

\subsubsection{Random 7-axis Paths for Performance Comparison}\label{sub2sec:random_7axis_path}

We randomly generate 100 7-axis spline paths $\vq=\vq(s)$ with pure kinematic constraints. For each path, 20 waypoints are uniformly sampled in $[-1,1]^7$ and interpolated using a uniform quintic spline. The $s$-domain $[0,1]$ is discretized into 1000 segments. The nominal maximal joint velocity, acceleration, and jerk are set to 1, 5, and 50, respectively. The final upper bounds are obtained by independently perturbing each nominal bound in each axis by a random factor uniformly sampled from $[0.9,1.1]$, while the corresponding lower bounds are set symmetrically.
Two different objectives are considered: time objective $J_1=t_\text{f}$ and time--energy hybrid objective $J_2=t_\text{f}+0.1\int_0^{t_\text{f}}\sum_{i=1}^7 \tau_i^2\,\mathrm{d}t$.
For a robot-independent numerical benchmark, we adopt the unit point-mass model $\vtau=\ddot{\vq}$, under which the thermal energy \eqref{eq:objective_thermal} reduces to the integral of squared accelerations. RDDP also accommodates the more general dynamics in \eqref{eq:robot_dynamics}.
The quantitative results for TOPP2, GOPP2, TOPP3, and GOPP3 are summarized in Table \ref{tab:performance_comparison}. Success rate (SR) is the fraction of cases returning a feasible trajectory. GOPP3-GDP reports infeasible when greedy pruning eliminates all candidate grid states at some $s_k$. QOCO's single GOPP3-SOCP failure returned \texttt{NumericalError}.

Table \ref{tab:performance_comparison} compares computational time ($T_\text{c}$), objective value ($J$), and success rate (SR). After the prescribed loops, the mean absolute gaps to SOCP are $1.99\times10^{-5}$, $0.0140$, $0.00183$, and $0.728$ for TOPP2, GOPP2, TOPP3, and GOPP3, respectively. Here, the OPP3 optimality gaps concern the conservative convexified subproblems. Across all four problem settings, RDDP achieves objective values comparable to SOCP with substantially lower mean computational time. Relative to the fastest SOCP variants in each problem class, RDDP achieves speedups of approximately 15.3 and 4.9 for OPP2 and OPP3, respectively.

\begin{figure}[!]
    \centering
    \input{figures/convergence.tikz}
    \caption{Relative optimality gap of GOPP2-IDP with varying state-grid numbers and our GOPP2-RDDP with varying DDP loop numbers. The objective value of GOPP2-SOCP is used as the reference. It demonstrates the faster optimality convergence of the RDDP passes over uniform grid densification.}
    \label{fig:Convergence}
\end{figure}

\begin{figure*}[!]
    \centering
    \resizebox{\textwidth}{!}{\input{figures/trajectories_random.tikz}}
    \caption{Representative normalized profiles for one instance. The upper and lower bounds are normalized to $\pm1$. (a) TOPP2. (b) GOPP2. (c) TOPP3. (d) GOPP3.}
    \label{fig:trajectories_random}
\end{figure*}

Note that once a practical time limit is imposed, GOPP2-SOCP${}^*$ and GOPP3-SOCP${}^*$ exhibit evident optimality degradation. Under the reported early-termination settings, RDDP approaches achieve better objective values and higher computational efficiency than Clarabel-based SOCP${}^*$ methods.

RA-based methods are fastest for TOPP2 and TOPP3, but their greedy forward pass limits general-objective optimality. In TOPP2, TOPP2-RA and GOPP2-RDDP have absolute optimality gaps of $2.05\times10^{-4}$ and $1.99\times10^{-5}$, respectively. With fixed discretization, TOPP2-RA cannot refine its solution, whereas GOPP2-RDDP can reduce the optimality gap by the backward and forward passes, as shown in Fig. \ref{fig:Convergence}. For TOPP3, GOPP3-RDDP has a significantly smaller objective difference compared to GOPP3-SOCP than TOPP3-RA. TOPP3-RA-IRC also reduces the computation time by about 60\% compared with TOPP3-RA, which demonstrates benefits beyond RDDP loops. In summary, RDDP's objective-aware value-function information supports both refinement and general objectives.

Under the tested settings, RDDP outperforms grid-based DP baselines in terms of computational time, optimality, and success rate. It avoids IDP's restriction to predefined state-grid points. Fig. \ref{fig:Convergence} shows that refining value functions over continuous reachable sets improves optimality much faster than uniformly densifying the IDP grid. This observation agrees with Section \ref{subsec:toy_example}.

\begin{figure*}[!]
    \centering
    \input{figures/dim_scaling.tikz}
    \caption{Scalability of different methods to $n$-axis paths. (a) Computational time in TOPP. (b) Relative objective difference in TOPP. (c) Computational time in GOPP. (d) Relative objective difference in GOPP. All SOCP results use Clarabel; their objective values serve as the references in (b) and (d).}
    \label{fig:dim_scaling}
\end{figure*}

\begin{figure*}[!]
    \centering
    \input{figures/real_trajecotries.tikz}
    \caption{Real-world experiments. (a) The robotic platform. (b) The normalized command and feedback signals in TOPP. (c)~The normalized command and feedback signals in GOPP. Both TOPP and GOPP are solved by the proposed RDDP approach.}
    \label{fig:real_trajecotries}
\end{figure*}

IDP's curse of dimensionality substantially increases computation time in TOPP3 and GOPP3. With a greedy beam search heuristic, GOPP3-GDP improves efficiency while sacrificing optimality and success rate. The feasibility of GOPP3-GDP is not theoretically guaranteed and is empirically sensitive to the numbers of control segments and sampled states. In contrast, GOPP3-RDDP achieves $100\%$ success, better objective values, and shorter computation time than DP baselines.

Fig. \ref{fig:trajectories_random} compares RDDP profiles for the same path and constraints. TOPP2 and TOPP3 follow the bang-singular-bang control, consistent with PMP \cite{pontryagin1987mathematical}. TOPP3 retains TOPP2's overall shape, while acceleration is smoother under jerk constraints. The time-energy objective gives GOPP2 and GOPP3 longer traversal times and smaller accelerations than TOPP profiles. These results demonstrate RDDP's flexibility.

\subsubsection{Scalability to $n$-axis Paths}

To examine the scalability of RDDP, we generate 10 random $n$-axis spline paths $\vq=\vq(s)$ with pure kinematic constraints for $n=5,10,\dots,100$. Path and constraint settings are the same as in Section \ref{sub2sec:random_7axis_path}. OPP2 and OPP3 have $2n+1$ and $4n+1$ inequality constraints at each $s_k$, respectively. Clarabel is the fixed SOCP solver in this study.

Fig. \ref{fig:dim_scaling} shows that computation time grows with $n$ for all methods, but much more slowly for RDDP than for the Clarabel-based SOCP baselines, which is attributed to the efficient low-dimensional DP decomposition of RDDP. At large $n$, GOPP3-RDDP approaches TOPP3-RA's computational time, which supports the efficiency of searching-based backward reachable-set computation. RDDP's relative objective difference to SOCP remains small as $n$ grows. Consistent with Section \ref{sub2sec:random_7axis_path}, these results demonstrate efficient scaling to $n$-axis paths with small objective difference.

\subsection{Real-World Experiments}\label{subsec:real_world_experiment}

Since SOCP and RDDP yield similar solutions for the same OPP problem, this subsection conducts a real-world case study to illustrate the practical relevance of general-objective OPP, rather than to re-compare different OPP solvers. As shown in Fig. \ref{fig:real_trajecotries}(a), a 7-axis Franka Emika Panda manipulator is utilized, whose dynamic model is calibrated with asymmetric Coulomb friction. A quintic spline path with 16 waypoints is discretized into 1000 $s$-intervals. All kinodynamic constraints are as in Section \ref{subsec:toy_example}. Note that the objective in GOPP is a hybrid of time and thermal energy. An external joint-space PID controller tracks the planned trajectories through Franka-Rust \cite{Jizhou2025FrankaRust} at 1 kHz, using joint-position and velocity feedback to generate torque commands with torque-rate limiting. Built-in link-side torque sensors measure the joint torques.

\begin{table}[!t]
    \centering
    \caption{Metrics for Real-World Feedback Signals.}
    \label{tab:real_world}
    \begin{tabular}{lll}
        \toprule
        Metric & TOPP3 & GOPP3 \\ \midrule
        Traversal Time $J_\text{time}$ (s) & 13.462 & 13.482 \\
        Normalized Thermal Energy $J_\text{thermal}$ (s) & 11.810 & 11.778 \\
        Normalized Total Variation of Acceleration & 720.614 & 88.862 \\ \bottomrule
    \end{tabular}
\end{table}

A time objective $J_{\text{time}}=\int_{0}^{t_\f}\mathrm{d}t$ and a time-thermal hybrid objective $J_{\text{time}}+5J_{\text{thermal}}=\int_{0}^{t_\f}\left(1+5\norm[2]{\frac{\vtau(t)}{\vtau^+}}^2\right)\,\mathrm{d}t$ are utilized for TOPP and GOPP, respectively. For compactness, Fig. \ref{fig:real_trajecotries}(b)-(c) shows representative profiles with constraint saturation, and Table \ref{tab:real_world} reports quantitative results. In this task, gravity dominates joint torques; hence, the thermal penalty in GOPP does not substantially reduce overall motion velocity or acceleration. GOPP's traversal time is only 0.15\% longer than that of TOPP, and their thermal energies are close.

The trajectory smoothness is significantly improved with the time-thermal objective in this test, as shown in Fig. \ref{fig:real_trajecotries}(b)-(c). Jerk bounds limit magnitude but do not penalize switching. In TOPP, the planned jerk remains feasible yet frequently alternates, accompanied by larger feedback-acceleration variations. In GOPP, the total variation of feedback joint acceleration, normalized by $\ddot{\vq}^{+}$, is reduced by 87.7\% compared with TOPP. Note that the total variation is not explicitly penalized in GOPP. This suggests that the thermal-energy term may reduce acceleration oscillations in the tested setting.

\section{Conclusion}

This paper has proposed RDDP for OPP with general objectives and constraints. Replacing the relatively complete recourse assumption with explicit reachability, RDDP guarantees global optimality for convex OPP and KKT convergence for non-convex OPP under OPP-specific conditions, with efficient instantiations for OPP2 and OPP3. Numerical experiments demonstrate feasibility and objective values comparable to CO baselines evaluated using 4 optimization solvers, with speedups of approximately 15.3 and 4.9 times over the fastest baseline for OPP2 and OPP3, respectively. Compared with RA and grid-based DP, RDDP better combines reachability-based feasibility preservation with objective-aware value-function guidance. Scalability to multi-axis paths is verified with up to 100 axes. In summary, RDDP addresses a key capability gap in OPP by unifying certifiable general-objective optimization, reachability-based feasibility preservation, and online-compatible low-dimensional DP computation in a single state-grid-free framework for both OPP2 and OPP3.

\section*{Acknowledgments}

The authors would like to thank Dr. Suqin He, Dr. Shize Lin, and Mr. Jun Chen for their expertise in TOPP.

\appendices

\section{Applicability of OPP Formulation \eqref{eq:op_discretize}}\label{app:applicability}

This appendix provides the construction of the OPP formulation \eqref{eq:op_discretize} and discusses its applicability to various OPP scenarios. Consider the description of OPP in the beginning of Section \ref{subsec:formulation_opp}. In the time domain, the state and control variables of $m$-th order OPP problems are $\hat{\vx}(t)=(s(t), \dot{s}(t), \ldots, s^{(m-1)}(t))\in\R^m$ and $\hat{u}(t)=s^{(m)}(t)\in\R$, respectively. Then, the $m$-th order OPP problem can be formulated as the following optimal control problem:
\begin{subequations}\label{eq:ocp_time}
    \begin{align}
        \min_{\hat{\vx}(\bullet), \hat{u}(\bullet), t_\text{f}} \quad &J = \int_0^{t_\text{f}} \widehat{L}(\hat{\vx}(t), \hat{u}(t)) \mathrm{d}t + \widehat\Phi(\hat{\vx}(t_\text{f})) \label{eq:ocp_time_objective}\\
        \st \quad & \dot{s}^{(i-1)}(t) = s^{(i)}(t), \,\,i=1,2,\ldots,m, \\
        &\dot{s}(t) > 0,\,t\in(0,t_\text{f}),\label{eq:ocp_time_dynamics}\\
        &\hat{\vf}(\hat{\vx}(t); s(t)) \leq \vzero,  \label{eq:ocp_time_state_constraints}\\
        &\hat{\vh}(\hat{\vx}(t), \hat{u}(t); s(t)) \leq \vzero, \label{eq:ocp_time_mix_constraints} \\
        &\hat{\vx}(0)\in\widehat{\mathcal{X}}_0, \,\, \hat{\vx}(t_\text{f})\in\widehat{\mathcal{X}}_\text{f}.\label{eq:ocp_time_boundary}
    \end{align}
\end{subequations}
In OPP \eqref{eq:ocp_time}, Eqs. \eqref{eq:ocp_time_dynamics}--\eqref{eq:ocp_time_mix_constraints} are set for $t\in(0,t_\text{f})$.

Following \cite{verscheure2009time,debrouwere2013time}, OPP \eqref{eq:ocp_time} can be transformed into equivalent problems in the $s$-domain to get rid of the free terminal time $t_\text{f}$ and non-convexity caused by the time variable. Denote the derivative of a variable $\bullet$ with respect to $s$ as $\bullet'=\frac{\mathrm{d}\bullet}{\mathrm{d}s}$. Let $x_1(s)=\frac12\dot{s}^2$ and $x_{i+1}(s)=x_i'(s)$ for $i=1,2,\ldots,m-1$. Then, the state and control variables in the $s$-domain are $\vx(s)=(x_i)_{i=1}^{m-1}\in\R^{m-1}$ and $u(s)=x_{m}(s)\in\R$, respectively. The $m$-th order OPP problem can be represented as the following optimal control problem in the $s$-domain:
\begin{subequations}\label{eq:ocp_parameter}
    \begin{align}
        \min_{\vx(\bullet), u(\bullet)} \quad &J = \int_0^{s_\text{f}} L(\vx(s), u(s); s) \mathrm{d}s + \Phi(\vx(s_\text{f})) \label{eq:ocp_parameter_objective}\\
        \st \quad & x_i'(s) = x_{i+1}(s), \,\,i=1,2,\ldots,m-1, \\
        &\vf(\vx(s); s) \leq \vzero,  \label{eq:ocp_parameter_state_constraints}\\
        &\vh(\vx(s), u(s); s) \leq \vzero, \label{eq:ocp_parameter_mix_constraints} \\
        &\vx(0)\in \mathcal{X}_0, \,\, \vx(s_\text{f})\in \mathcal{X}_\text{f}.\label{eq:ocp_parameter_boundary}
    \end{align}
\end{subequations}
Specifically, the constraint $x_1(s)\geq0$, $\forall s\in(0,s_\text{f})$ is contained in the state constraints \eqref{eq:ocp_parameter_state_constraints}.

Generally, OPP2 \cite{verscheure2009time,pham2018new} and OPP3 \cite{debrouwere2013time,wang2026online} problems are of interest. Denote
\begin{equation*}
    x_2(s)=x_1'(s)=\frac{\dot{x}_1}{\dot{s}}=\ddot{s},\,\,x_3(s)=x_2'(s)=\frac{\dot{x}_2}{\dot{s}}=\frac{\dddot{s}}{\dot{s}}.
\end{equation*}
For OPP2, the decision variables are $\vx(s)=x_1(s)$ and $u(s)=x_2(s)$. For OPP3, the decision variables are $\vx(s)=(x_1(s), x_2(s))$ and $u(s)=x_3(s)$.

It has been recognized that OPP \eqref{eq:ocp_parameter} provides a highly expressive framework, capable of modeling various objectives and constraints. Some examples are provided as follows.

\begin{example}
    Consider the joint torque constraint $\underline{\vtau}\leq\vtau(s)\leq\overline{\vtau}$. The robot dynamics are given by
    \begin{equation}\label{eq:robot_dynamics}
        \vtau=\mM(\vq)\ddot{\vq}+\mC(\vq,\dot{\vq})\dot{\vq}+\vg(\vq)+\vtau_{\mathrm{c}}(\dot{\vq})+\vtau_{\mathrm{v}}(\dot{\vq}),
    \end{equation}
    where $\mM(\vq)$, $\mC(\vq,\dot{\vq})$, and $\vg(\vq)$ are the inertia matrix, the Coriolis and centrifugal matrix, and the gravity vector, respectively. Specifically, $\mC(\vq,\dot{\vq})$ is linear w.r.t. $\dot{\vq}$. The Coulomb friction $\vtau_{\mathrm{c}}$ and viscous friction $\vtau_{\mathrm{v}}$ allow directional asymmetry and are defined componentwise as
    \[
        \left(\tau_{\mathrm{c},i}(\dot\vq),\tau_{\mathrm{v},i}(\dot\vq)\right):=
        \begin{cases}
            \left(\tau_{\mathrm{c},i}^{+},F_{\mathrm{v},i}^+\dot q_i\right), & \dot q_i>0,\\
            (0,0), & \dot q_i=0,\\
            \left(-\tau_{\mathrm{c},i}^{-},F_{\mathrm{v},i}^-\dot q_i\right), & \dot q_i<0,
        \end{cases}
    \]
    where $\tau_{\mathrm{c},i}^\pm,F_{\mathrm{v},i}^{\pm}\geq0$. These coefficients are zero when the corresponding friction is omitted from the model. Note that
    \begin{subequations}
        \begin{align}
            &\dot{\vq}=\vq'\dot{s}=\vq'\sqrt{2x_1},\label{eq:velocity}\\
            &\ddot{\vq}=\vq''\dot{s}^2+\vq'\ddot{s}=2\vq''x_1+\vq'x_2.\label{eq:acceleration}
        \end{align}
    \end{subequations}
    For $\dot{s}>0$ along the path $\vq(s)$, the torque can be expressed as
    \begin{align*}
        \vtau(s)=&\mM(\vq)(2\vq''x_1+\vq'x_2)+2\mC(\vq,\vq')\vq'x_1+\vg(\vq)\notag\\
        &+\vtau_{\mathrm{c}}(\vq')+\vtau_{\mathrm{v}}(\vq')\sqrt{2x_1}.
    \end{align*}
    Without viscous friction, the torque remains affine in $x_1$ and $x_2$, and its box constraints $\underline{\vtau}\leq\vtau\leq\overline{\vtau}$ remain linear. With viscous friction, the additional $\sqrt{2x_1}$ term generally makes the torque constraints non-convex.
    In OPP2, the torque constraint can be represented as the mixed constraints \eqref{eq:ocp_parameter_mix_constraints}. In OPP3, it can be represented as the state constraints \eqref{eq:ocp_parameter_state_constraints}.

    Similar analysis applies
    to the box constraints on joint velocity \eqref{eq:velocity} and acceleration \eqref{eq:acceleration}, both of which can be transformed into linear inequalities in $x_1$ and $x_2$.
\end{example}

\begin{example}
    For OPP3, the constraint on the joint jerk $\dddot{\vq}$ can be represented as the mixed constraints \eqref{eq:ocp_parameter_mix_constraints}. Since
    \begin{equation*}
        \dddot{\vq}=\vq'''\dot{s}^3 + 3\vq''\dot{s}\ddot{s} + \vq'\dddot{s} = \sqrt{2x_1}(2\vq'''x_1 + 3\vq''x_2 + \vq'x_3),
    \end{equation*}
    the jerk constraint $\underline{\dddot{\vq}}\leq \dddot{\vq}\leq \overline{\dddot{\vq}}$ can be represented as
    \begin{equation}\label{eq:jerk_constraint}
        \frac{1}{\sqrt{2x_1}}\underline{\dddot{\vq}} \leq 2\vq'''x_1 + 3\vq''x_2 + \vq'x_3 \leq \frac{1}{\sqrt{2x_1}}\overline{\dddot{\vq}}.
    \end{equation}
    In \eqref{eq:jerk_constraint}, the limits of jerk usually satisfy $\underline{\dddot{\vq}}<\vzero<\overline{\dddot{\vq}}$. Then, the jerk constraint is concave in $x_1$. Similar analysis can be applied to the constraint on the joint torque rate $\dot{\vtau}$ \cite{debrouwere2013time,wang2026online}.
\end{example}

\begin{example}
    \begin{subequations}
        Consider some common objectives. In TOPP, the objective is to minimize the terminal time
        \begin{equation}\label{eq:objective_time}
            J_\text{time}=\int_0^{t_\text{f}} \mathrm{d}t = \int_0^{s_\text{f}} \frac{1}{\dot{s}} \mathrm{d}s = \int_0^{s_\text{f}} \frac{1}{\sqrt{2x_1}} \mathrm{d}s.
        \end{equation}
        \cite{verscheure2009time} further considers the thermal energy
        \begin{equation}\label{eq:objective_thermal}
            J_\text{thermal} = \int_0^{t_\text{f}} \vtau^\top \mQ\vtau \mathrm{d}t = \int_0^{s_\text{f}} \frac{\vtau^\top\mQ\vtau}{\sqrt{2x_1}} \mathrm{d}s,\,\mQ\succeq \vzero,
        \end{equation}
        and the total variation of the torque
        \begin{equation}\label{eq:objective_torque_variation}
            J_\text{variation} = \int_0^{t_\text{f}} \mathrm{d}\norm[]{\vtau(t)}= \int_0^{s_\text{f}}\mathrm{d} \norm[]{\vtau(s)} .
        \end{equation}
    \end{subequations}
    In OPP2, a linear combination of \eqref{eq:objective_time} and \eqref{eq:objective_thermal} can be represented as the objective \eqref{eq:ocp_parameter_objective}. In OPP3, a linear combination of all three objectives can be handled. The objectives $J_\text{thermal}$ and $J_\text{variation}$ are convex without viscous friction but generally non-convex when viscous friction is included.
\end{example}

Finally, OPP \eqref{eq:ocp_parameter} can be discretized into the formulation \eqref{eq:op_discretize}. All inequality constraints are enforced at the grid points $s_k$ or the two-sided limits $s_k^\pm$. Specifically, the state constraints \eqref{eq:op_discretize_state_constraints} refer to $\vf(\vx(s_k); s_k)=\vf(\vx_k; s_k)\leq\vzero$ in \eqref{eq:ocp_parameter_state_constraints}, whereas the mixed constraints \eqref{eq:op_discretize_mix_constraints} refer to $\vh(\vx(s_k), u(s_k^+); s_k)=\vh(\vx_k, u_k; s_k)\leq\vzero$ for $k<N$ and $\vh(\vx(s_k), u(s_k^-); s_k)=\vh(\vA_{k-1}\vx_{k-1}+\vb_{k-1}u_{k-1}, u_{k-1}; s_k)\leq\vzero$ for $k>0$. Furthermore, the objective function $L_k$ in \eqref{eq:op_discretize_objective} is a numerical approximation of the integral $\int_{s_{k}}^{s_{k+1}} L(\vx(s), u(s); s) \mathrm{d}s$ in \eqref{eq:ocp_parameter_objective}.

In this paper, feasibility guarantees concern the discretized problem \eqref{eq:op_discretize}. For OPP2 \cite[Appendix D2]{pham2018new} and OPP3 \cite[Theorem 1]{wang2026online}, enforcing constraints at both endpoints of each interval yields an $\mathcal{O}(\max_k\Delta_k^2)$ bound on inter-sample constraint violations under the assumptions stated in these literatures.

\section{Proofs of Theorems and Propositions}\label{app:proofs}

\begin{proof}[Proof of Lemma \ref{lemma:convexity_reachset_valfunc}]
    By Assumption \ref{assum:feasibility_smoothness}, the convex set
    \begin{align*}
        \widetilde{\mathcal{B}}_k=\big\{&(\vx_{k}, \ldots, \vx_N, u_k, \ldots, u_{N-1})\mid \notag\\
        &\forall k\leq j<N,\,\vx_{j+1} = \vA_j\vx_j+\vb_j u_j, \notag\\
        &\vf_j(\vx_j) \leq \vzero,\,\vh_j(\vx_j, u_j) \leq \vzero,\,\vx_N \in \mathcal{X}_\text{f} \big\}
    \end{align*}
    is compact. As its projection, $\mathcal{B}_k$ is also convex and compact. By \cite[Section 3.2.5]{boyd2004convex}, $V_k^\star$ is convex in $\mathcal{B}_k$ since the map $\widetilde{\mathcal{B}}_k\to\R$, $\left((\vx_j)_{j=k}^N, (u_j)_{j=k}^{N-1}\right)\mapsto\sum_{j=k}^{N-1} L_j(\vx_j, u_j) + \Phi(\vx_N)$ is convex. The convexity and compactness of $\mathcal{C}_k(\vx_k)$ are straightforward. If all constraints are linear and $\mathcal{X}_\text{f}$ is a polytope, then $\widetilde{\mathcal{B}}_k$ is a polytope; hence, $\mathcal{B}_k$ is also a polytope.
\end{proof}

\begin{proof}[Proof of Proposition \ref{prop:Lipschitz}]
    We prove this proposition by induction. For $k=N$, \eqref{eq:approximation_terminal_value} and Assumption \ref{assum:feasibility_smoothness} imply that $\mathcal{V}_N^p=\Phi$ is convex and locally Lipschitz continuous in $\mathcal{B}_N=\mathcal{X}_\text{f}$. By the finite subcover property of $\mathcal{B}_N$, $\mathcal{V}_N^p=\Phi$ is Lipschitz continuous in $\mathcal{B}_N$ with the constant $C_N^V>0$ independent of $p$ and the cut-generation sets. By \cite[Corollary 3.1]{nguyen2021lipschitz}, for each $\vx\in\mathcal{B}_N$, there exists $\vg\in\partial\mathcal{V}_N^p(\vx)$ such that $\norm[2]{\vg}\leq C_N^V$.

    For $k<N$, assume that the conclusion holds for $\mathcal{V}_{k+1}^p$. By Assumption \ref{assum:feasibility_smoothness} and Lemma \ref{lemma:convexity_reachset_valfunc}, the convex set $\Gamma_k:=\{(\vx,u):\vx\in\mathcal{B}_k,\,u\in\mathcal{C}_k(\vx)\}$ is compact. The finite subcover property of $\Gamma_k$ implies the Lipschitz continuity of $L_k$ on $\Gamma_k$; hence, $\widehat{J}_k(u;\vx)=L_k(\vx, u) + \mathcal{V}_{k+1}^p(\vA_k\vx+\vb_k u)$ is Lipschitz continuous on $\Gamma_k$, where the Lipschitz constant $C_k^J$ is independent of $p$ and cut-generation sets. Moreover, by \cite[Section 3.2.5]{boyd2004convex}, $\widehat{\mathcal{V}}_k^p(\vx)=\min\{\widehat{J}_k(u;\vx)|\underbar{u}_k(\vx)\leq u\leq \bar{u}_k(\vx)\}$ is convex.

    We now prove that $\widehat{\mathcal{V}}_k^p$ is Lipschitz on $\mathcal{B}_k$ with a Lipschitz constant independent of $p$. For any $\vx,\vy\in\mathcal{B}_k$, choose an optimal control $u_\vx^\star\in\mathcal{C}_k(\vx)$ at $\vx$ and set a feasible control $u_\vy:=\min\{\bar{u}_k(\vy),\max\{\underbar{u}_k(\vy),u_\vx^\star\}\}\in\mathcal{C}_k(\vy)$ at $\vy$. By the finite subcover property of $\mathcal{B}_k$ and Assumption \ref{assum:Lipschitz}, $\underbar{u}_k$ and $\bar{u}_k$ have a Lipschitz constant $C_k^u>0$ on $\mathcal{B}_k$, independent of $p$ and cut-generation sets. Since $\abs{u_\vy-u_\vx^\star}\leq C_k^u\norm[2]{\vy-\vx}$,
    \begin{align*}
        \widehat{\mathcal{V}}_k^p(\vy)-\widehat{\mathcal{V}}_k^p(\vx)
        \leq&\widehat{J}_k(u_\vy;\vy) - \widehat{J}_k(u_\vx^\star;\vx)\notag\\
        \leq&C_k^J\left(\abs{u_\vy-u_\vx^\star}+\norm[2]{\vy-\vx}\right)\notag\\
        \leq&C_k^J\left(C_k^u+1\right)\norm[2]{\vy-\vx}.
    \end{align*}
    Then, $C_k^V:= C_k^J(C_k^u+1)$ is independent of $p$ and cut-generation sets. For the same reason, $\widehat{\mathcal{V}}_k^p(\vx)-\widehat{\mathcal{V}}_k^p(\vy)\leq C_k^V\norm[2]{\vy-\vx}$; hence, $\widehat{\mathcal{V}}_k^p$ is $C_k^V$-Lipschitz continuous on $\mathcal{B}_k$.

    By \cite[Corollary 3.1]{nguyen2021lipschitz}, for each $\vx\in\mathcal{B}_k$, there exists $\vg\in\partial\widehat{\mathcal{V}}_k^p(\vx)$ such that $\norm[2]{\vg}\leq C_k^V$. Thus, a subgradient satisfying this bound can be selected at every sample. For any such selections, \eqref{eq:support_plane} implies that $\mathcal{V}_k^p$ is also convex and $C_k^V$-Lipschitz continuous on $\mathcal{B}_k$. Applying the same analysis gives the conclusion of $V_k^\star$.
\end{proof}

\begin{proof}[Proof of Proposition \ref{prop:bound_of_value_function}]
    According to Proposition \ref{prop:Lipschitz}, $\widehat{\mathcal{V}}_k^p$ is convex. In the definition \eqref{eq:support_plane}, for each $i$, we have $\widehat{\mathcal{V}}_k^p(\vx_k)\geq\widehat{\mathcal{V}}_k^p(\vx_k^i)+{\vg_k^i}^\top(\vx_k-\vx_k^i)$. Therefore, $\widehat{\mathcal{V}}_k^p(\vx_k)\geq\mathcal{V}_k^p(\vx_k)$ holds. According to \eqref{eq:def_value_function}, $V_k^\star(\vx_k)\leq J_k(\bx_k,\bu_k)$ also holds.

    Next, we prove that $\widehat{\mathcal{V}}_k^p(\vx_k)\leq V_k^\star(\vx_k)$ by induction. For $k=N-1$, $\mathcal{V}_N^p(\vx_N)=V_N^\star(\vx_N)=J_N(\bx_N,\bu_N)$ holds. By comparing \eqref{eq:def_value_function} and \eqref{eq:approximate_value_function_backward}, we have $\widehat{\mathcal{V}}_{N-1}^p(\vx_{N-1})=V_{N-1}^\star(\vx_{N-1})$. For $k<N-1$, assume that $\widehat{\mathcal{V}}_{k+1}^p(\vx_{k+1})\leq V_{k+1}^\star(\vx_{k+1})$ holds for each $\vx_{k+1}\in\mathcal{B}_{k+1}$. For any $\vx_k\in\mathcal{B}_k$, we have $\mathcal{V}_{k}^p(\vx_{k})\leq V_{k}^\star(\vx_{k})$ since $\mathcal{V}_{k+1}^p(\vx_{k+1})\leq\widehat{\mathcal{V}}_{k+1}^p(\vx_{k+1})\leq V_{k+1}^\star(\vx_{k+1})$. By comparing \eqref{eq:def_value_function} and \eqref{eq:approximate_value_function_backward}, we have $\widehat{\mathcal{V}}_k^p(\vx_k)\leq V_k^\star(\vx_k)$.
\end{proof}

\begin{proof}[Proof of Theorem \ref{thm:near_optimality_gap}]
    Evidently, $J_0(\bx_0^{p\star},\bu_0^{p\star}) \geq J^\star$ holds. Arbitrarily select a globally optimal solution $\bx_0^\star=(\vx_k^\star)_{k=0}^N$ and $\bu_0^\star=(u_k^\star)_{k=0}^{N-1}$ of COPP \eqref{eq:op_discretize}. By \eqref{eq:refine_x0}, we have $\mathcal{V}_0^p(\vx_0^{p\star}) \leq \mathcal{V}_0^p(\vx_0^\star)$. By Proposition \ref{prop:bound_of_value_function}, we have $\mathcal{V}_0^p(\vx_0^{\star}) \leq V_0^\star(\vx_0^{\star}) = J^\star$. Hence, we have $\varepsilon_\text{gap}(\bx_0^{p\star},\bu_0^{p\star}) = J_0(\bx_0^{p\star},\bu_0^{p\star}) - \mathcal{V}_0^p(\vx_0^{p\star}) \geq J_0(\bx_0^{p\star},\bu_0^{p\star}) - \mathcal{V}_0^p(\vx_0^{\star}) \geq J_0(\bx_0^{p\star},\bu_0^{p\star}) - J^\star \geq 0$.
\end{proof}

\begin{proof}[Proof of Theorem \ref{thm:adrp_optimal}]
    Consider the solution $\{(\bx_0^{p\star},\bu_0^{p\star})\}_{p=0}^\infty$ generated by RDDP. Assume that
    $$\varepsilon_\infty:=\limsup_{p\to\infty}\varepsilon_\text{gap}(\bx_0^{p\star},\bu_0^{p\star})>0.$$
    By Assumption \ref{assum:feasibility_smoothness}, the feasible set of COPP \eqref{eq:op_discretize} is compact. Take a convergent subsequence $\{(\bx_0^{p_i\star},\bu_0^{p_i\star})\}_{i=0}^\infty$ with $\limsup_{i\to\infty}\varepsilon_\text{gap}(\bx_0^{p_i\star},\bu_0^{p_i\star})=\varepsilon_\infty>0$. According to Proposition \ref{prop:Lipschitz}, for each $k$, convex functions $\widehat{\mathcal{V}}_k^p$ and $\mathcal{V}_k^p$ share the same Lipschitz constant $C_k^V$ across $p$. Consider any $i>j$. According to the RDDP sampling process \eqref{eq:ardp_more_sample}, for each $k$, the state $\vx_k^{p_j\star}\in\mathcal{B}_k$ has been sampled to construct $\mathcal{V}_k^{p_i}$ in \eqref{eq:support_plane}. Therefore, we have
        $\widehat{\mathcal{V}}_k^{p_i}(\vx_k^{p_j\star})=\mathcal{V}_k^{p_i}(\vx_k^{p_j\star})$.
    According to \eqref{eq:widehat_V_L_V_refine} and Proposition \ref{prop:Lipschitz}, we have
    \begin{align*}
        &\varepsilon_\text{gap}(\bx_0^{p_i\star},\bu_0^{p_i\star})\notag\\
        =&\Phi(\vx_N^{p_i\star}) - \mathcal{V}_0^{p_i}(\vx_0^{p_i\star})+\sum_{k=0}^{N-1} L_k(\vx_k^{p_i\star}, u_k^{p_i\star})\notag\\
        =&\sum_{k=0}^{N-1}\left[L_k(\vx_k^{p_i\star}, u_k^{p_i\star})+\mathcal{V}_{k+1}^{p_i}(\vx_{k+1}^{p_i\star})-\mathcal{V}_k^{p_i}(\vx_k^{p_i\star})\right]\notag\\
        =&\sum_{k=0}^{N-1}\left[\widehat{\mathcal{V}}_k^{p_i}(\vx_k^{p_i\star})-\mathcal{V}_k^{p_i}(\vx_k^{p_i\star})\right]\notag\\
        \leq&2\sum_{k=0}^{N-1}C_k^V\norm[2]{\vx_k^{p_j\star}-\vx_k^{p_i\star}}.
    \end{align*}
    Taking $j=i-1$ and letting $i\to\infty$ gives
    \[
        \varepsilon_\infty=\limsup_{i\to\infty}\varepsilon_\text{gap}(\bx_0^{p_i\star},\bu_0^{p_i\star})=0<\varepsilon_\infty,
    \]
    which leads to a contradiction. Therefore, $\varepsilon_\infty=0$. Applying Theorem \ref{thm:near_optimality_gap} gives $\lim_{p\to\infty}J_0(\bx_0^{p\star},\bu_0^{p\star})=J^\star$. So \eqref{eq:convergence_condition} holds.

    Moreover, consider any optimal solution $(\bx_0^\star,\bu_0^\star)$. Note that
    \begin{align*}
        \varepsilon_\text{gap}(\bx_0^{\star},\bu_0^{\star})
        =&J^\star-\mathcal{V}_0^p(\vx_0^\star)
        \leq J_0(\bx_0^{p\star},\bu_0^{p\star}) - \mathcal{V}_0^p(\vx_0^{p\star}).
    \end{align*}
    By \eqref{eq:dual_gap} and \eqref{eq:convergence_condition}, letting $p\to\infty$ gives $\varepsilon_\text{gap}(\bx_0^{\star},\bu_0^{\star})\to0$.
\end{proof}

\begin{proof}[Proof of Theorem \ref{thm:general_KKT}]
    For notational simplicity, denote $\by^{(q)}:=\big(\bx_0^{(q)},\bu_0^{(q)}\big)$ and $\by^\star:=(\bx_0^\star,\bu_0^\star)$. Assume that $\mathcal{F}_{\text{c}}^{(q)}\neq\varnothing$ and, for $q\geq1$, that $\by^{(q)}\in\mathcal{F}_{\text{c}}^{(q)}$. If $q\geq1$ and $\by^{(q)}$ is optimal in $\mathcal{P}_{\text{c}}^{(q)}$, then Theorem \ref{thm:adrp_optimal} implies that the solution $\by^{(q+1)}=\by^{(q)}$ satisfies \eqref{eq:requirement_KKT_convergence} after finitely many backward and forward passes; otherwise, Theorem \ref{thm:adrp_optimal} implies that RDDP can generate $\by^{(q+1)}\in\mathcal{F}_{\text{c}}^{(q)}\subset\mathcal{F}$ which also satisfies \eqref{eq:requirement_KKT_convergence}. By Definition \ref{def:DC_decomposable}, $\by^{(q+1)}\in\mathcal{F}_{\text{e}}^{(q+1)}$; hence, \eqref{eq:requirement_GOPP_feasible} implies $\by^{(q+1)}\in\mathcal{F}_{\text{c}}^{(q+1)}$. Therefore, $\mathcal{F}_{\text{c}}^{(q+1)}\neq\varnothing$. By induction, Theorem \ref{thm:general_KKT}--\ref{subthm:general_KKT_feasible} holds.

    By Definition \ref{def:DC_decomposable}, for all $q\geq1$, we have $J_{0}^{(q)}\big(\by^{(q+1)}\big)\geq J_{0}\big(\by^{(q+1)}\big)$ and $J_{0}^{(q)}\big(\by^{(q)}\big)= J_{0}\big(\by^{(q)}\big)$. Therefore, Theorem \ref{thm:general_KKT}--\ref{subthm:general_KKT_decresing} follows immediately from \eqref{eq:requirement_KKT_convergence_Jq_decrease}.

    Consider any convergent subsequence $\big\{\by^{(q_i)}\big\}_{i=1}^\infty$. Since $\mathcal{F}$ is compact, $\by^\star:=\lim_{i\to\infty}\by^{(q_i)}$ is feasible. Assume that $\by^\star$ is not optimal in $\mathcal{P}_{\text{e}}^\star$. For the convexified objective $J_0^\star$ at $\by_0^\star$, there exists another feasible point $\widetilde{\by}$ of $\mathcal{P}_{\text{e}}^\star$, satisfying the strict inequalities in Definition \ref{def:slater}, such that $\norm[2]{\widetilde{\vx}_k-\vx_k^\star}<\frac{\delta}{2}$ and
    \[
    J_0^\star(\widetilde{\by})<J_0^\star(\by^\star)=J_0(\by^\star)=\lim_{q\to\infty}J_0\big(\by^{(q)}\big).
    \]
    By Definition \ref{def:DC_decomposable}, all convexified constraints are jointly continuous w.r.t. their arguments and reference points. Let $I\in\mathbb{N}$ be sufficiently large such that for all $i\geq I$, $\widetilde{\by}\in\mathcal{F}_{\text{e}}^{(q_i)}$ and $\norm[2]{\vx_k^{(q_i)}-\vx_k^\star}<\frac{\delta}{2}$. Then, $\norm[2]{\vx_k^{(q_i)}-\widetilde{\vx}_k}<\delta$; hence, applying \eqref{eq:KKT_approximate_reachable_set_condition} recursively backward implies $\widetilde{\by}\in\mathcal{F}_{\text{c}}^{(q_i)}$. By \eqref{eq:dual_gap}, \eqref{eq:requirement_KKT_convergence_epsilon_q}, and Theorem \ref{thm:near_optimality_gap}, we have
    \[
    J_0\big(\by^{(q_i+1)}\big)\leq J_0^{(q_i)}\big(\by^{(q_i+1)}\big)\leq J_0^{(q_i)}(\widetilde{\by})+\varepsilon_{q_i}.
    \]
    Letting $i\to\infty$ gives $J_0(\by^\star)\leq J_0^{\star}(\widetilde{\by})$, which leads to a contradiction. Therefore, $\by^\star$ is optimal in $\mathcal{P}_{\text{e}}^\star$.

    By Definition \ref{def:slater} and \cite[Section 5.5.3]{boyd2004convex}, $\by^\star$ admits multipliers satisfying the KKT conditions of $\mathcal{P}_{\text{e}}^\star$. According to Definition \ref{def:DC_decomposable}, convexification preserves the objective gradient, the constraint values, and the constraint gradients at the reference point, while all affine equalities remain unchanged. Therefore, the same multipliers satisfy the KKT conditions of the original GOPP \eqref{eq:op_discretize} at $\by^\star$, which proves Theorem \ref{thm:general_KKT}--\ref{subthm:general_KKT_convergence}.
\end{proof}

\ifCLASSOPTIONcaptionsoff
    \newpage
\fi

\bibliographystyle{myIEEEtran}
\bibliography{IEEEabrv,refs/ref}

\end{document}